\RequirePackage{fix-cm}
\documentclass[smallextended]{svjour3}       % onecolumn (second format)

\usepackage{graphicx}
\usepackage{amssymb}
\usepackage{amsmath}
\usepackage{amscd,amssymb,euscript,graphics}
\newlength{\figwidth}
\newlength{\figheight}
\newlength{\philwidth}

\def\mbb#1{\mathbb{#1}}%           blackboard font
\def\lrs#1{\ensuremath{\left[{#1}\right]}}%
\def\lrv#1{\ensuremath{\lvert{#1}\rvert}}%
\def\blrv#1{\ensuremath{\Bigg\lvert{#1}\Bigg\rvert}}%

\newcommand{\R}{\ensuremath{\mathbb{R}}}

\newcommand{\norm}[1]{\ensuremath{\|#1\|}}

\newcommand{\tr}{\mathrm{Tr}}

\newcommand{\NN}{{\mathbb N}}

\newcommand{\ZZ}{{\mathbb Z}_+}

\newcommand{\babs}[1]{\blrv{#1}}
\newcommand{\abs}[1]{\lrv{#1}}%

\def\Prblr#1{\mbb{P}\kern-2pt\lrs{{#1}}}%

\newcommand{\prob}[1]{{\mathbb P}\Big(#1\Big)}
\newcommand{\expect}[1]{{\mathbb E}\left[#1\right]}

\newcommand{\euclidnorm}[1]{\lvert \lvert #1 \rvert \rvert}
\newcommand{\spnorm}[1]{\lvert \lvert #1 \rvert \rvert_1}
\newcommand{\maxnorm}[1]{\lvert \lvert #1 \rvert \rvert_\infty}
\newcommand{\opnorm}[1]{\lvert \lvert #1 \rvert \rvert_2}
\newcommand{\lc}{\left\{}
\newcommand{\rc}{\right\}}
\newcommand{\lb}{\left(}
\newcommand{\rb}{\right)}
\newcommand{\ls}{\left[}
\newcommand{\rs}{\right]}
\newcommand{\cmplx}{\mathbb C}
\newcommand{\re}{{\mathbb R}}

\newcommand{\tdvarphi}{\tilde{\varphi}}

\newcommand{\tp}{\tilde{\psi}}
\newcommand{\Wc}{W^o}

\newtheorem{cor}{Corollary}
\newtheorem{prop}{Proposition}
\begin{document}

\title{Asymptotic Behavior of the Maximum and Minimum Singular Value of Random Vandermonde Matrices}
%\subtitle{Do you have a subtitle?\\ If so, write it here}

\titlerunning{Asymptotic Behavior of the Maximum and Minimum Singular Value}        % if too long for running head

\author{Gabriel H. Tucci         \and
        Philip A. Whiting %etc.
}

%\authorrunning{Short form of author list} % if too long for running head

\institute{Gabriel H. Tucci  \at
		   Alcatel--Lucent, 600 Mountain Ave\\ 
		   Murray Hill, NJ 07974\\
           \email{gabriel.tucci@alcatel-lucent.com}           %  \\
           \and
           Philip A. Whiting \at
           Alcatel--Lucent, 600 Mountain Ave\\ 
		   Murray Hill, NJ 07974\\
           \email{philip.whiting@alcatel-lucent.com} 
}

\date{Received: date / Accepted: date}
% The correct dates will be entered by the editor

\maketitle

\begin{abstract}
This work examines various statistical distributions in connection with random Vandermonde matrices and their extension to $d$--dimensional phase distributions. Upper and lower bound asymptotics for the maximum singular value are found to be $O(\log^{1/2}{N^{d}})$ and $\Omega((\log N^{d} /(\log \log N^d))^{1/2})$ respectively where $N$ is the dimension of the matrix, generalizing the results in \cite{TW}. We further study the behavior of the minimum singular value of these random matrices. In particular, we prove that the minimum singular value is at most $N\exp(-C\sqrt{N}))$ with high probability where $C$ is a constant independent on $N$. Furthermore, the value of the constant $C$ is determined explicitly. The main result is obtained in two different ways. One approach uses techniques from stochastic processes and in particular, a construction related to the Brownian bridge. The other one is a more direct analytical approach involving combinatorics and complex analysis. As a consequence, we obtain a lower bound for the maximum absolute value of a random complex polynomial on the unit circle, which may be of independent mathematical interest. Lastly, for each sequence of positive integers $\{k_p\}_{p=1}^{\infty}$ we present a generalized version of the previously discussed matrices. The classical random Vandermonde matrix corresponds to the sequence $k_{p}=p-1$. We find a combinatorial formula for their moments and we show that the limit eigenvalue distribution converges to a probability measure supported on $[0,\infty)$. Finally, we show that for the sequence $k_p=2^{p}$ the limit eigenvalue distribution is the famous Marchenko--Pastur distribution. 
% \PACS{PACS code1 \and PACS code2 \and more}
\end{abstract}

\section{Introduction}\label{intro}%

Large dimensional random matrices are of much interest in statistics, where they play a role in multivariate analysis. In his seminal paper, Wigner \cite{wigner} proved that the spectral measure of a wide class of symmetric random matrices of dimension $N$ converges, as $N\to\infty$, to the semicircle law. Much work has since been done on related random matrix ensembles, either composed of (nearly) independent entries, or drawn according to weighted Haar measures on classical groups (e.g., orthogonal, unitary, simplectic). The limiting behavior of the spectrum of such matrices is of considerable interest for mathematical physics and information theory. In addition, such random matrices play an important role in operator algebra studies initiated by Voiculescu, now known as free (non--commutative) probability theory (see \cite{Voi1} and \cite{Voi2} and the many references therein). The study of large random matrices is also related to interesting questions in combinatorics, geometry, algebra and number theory. More recently, the study of large random matrices ensembles with additional structure have been considered. For instance, the properties of the spectral measures of random Hankel, Markov and Toeplitz matrices with independent entries have been studied in \cite{dembo}. 

In this paper we study several aspects of random Vandermonde matrices with unit magnitude complex entries and their generalizations. An $N\times L$ matrix ${\bf V}$ with unit complex entries
is a {\em Vandermonde matrix} if there exist values $\theta_1,\ldots,\theta_L \in [0,1]$ such that
\begin{equation}
\label{eqn_Vandermondedefn}
{\bf V} := \frac{1}{\sqrt{N}}\,
\lb
\begin{array}{lcl}
1              & \ldots & 1 \\
e^{2\pi i\theta_1} & \ldots & e^{2\pi i\theta_L} \\
\vdots         & \ddots & \vdots \\
e^{2\pi i(N-1)\theta_1} & \ldots & e^{2\pi i(N-1)\theta_L} 
\end{array}
\rb
\end{equation}
(see \cite{GC02} or \cite{TW} for more details). A random Vandermonde matrix is produced if the entries of the phase vector $\theta:= \lb \theta_1,\ldots, \theta_L \rb \in [0,1]^L$ are random variables. For the purposes of this paper we assume that the phase vector has i.i.d. components with an absolutely continuous distribution $\nu$.

\par Vandermonde matrices were defined in \cite{SP} and were also called $d$--fold Vandermonde matrices. The case $d=1$ are the matrices in (\ref{eqn_Vandermondedefn}). For $d\geq 2$, these matrices are defined by selecting $L$ random vectors $x_q$ independently in the $d$--dimensional hypercube $[0,1]^{d}$. These vectors are called the vectors of phases. Given a scale parameter $N$, consider the function defined by $\gamma: \lc 0,1,\ldots,N-1 \rc^d \to \lc 0,1,\ldots,N^{d}-1 \rc$ such that for every vector of integers $\ell = (\ell_1,\ell_2,\ldots,\ell_d) \in \lc 0,1,\ldots,N-1 \rc^d$ the value $\gamma(\ell)$ is equal to
$$
\gamma(\ell) := \sum_{j=1}^{d}{N^{j-1}\ell_{j}}.
$$
It is easy to see that this function is a bijection over the set $\lc 0,1,\ldots,N^{d}-1 \rc$. Now we define the $N^{d}\times L$ matrix ${\bf V}^{(d)}$ as
\begin{equation}
{\bf V}^{(d)}_{(\gamma(\ell),q)} := \frac{1}{N^{d/2}}\, \mathrm{exp}\Big(2\pi i \langle \ell,x_q\rangle\Big).
\label{eqn_vandergendef} 
\end{equation}
For the case $d=1$ we drop the upper index and denote this matrix by ${\bf V}$; for $d\geq 2$ we use ${\bf V}^{(d)}$.

Random Vandermonde matrices and their extended versions are a natural construction with a wide range of applications in fields as diverse as finance \cite{Norberg}, signal processing \cite{SP}, wireless communications \cite{Porst}, statistical analysis \cite{Anderson}, security \cite{Sampaio} and biology \cite{Strohmer}. This stems from the close relationship that unit magnitude complex Vandermonde matrices have with the discrete Fourier transform.  Among these, there is an important recent application for signal reconstruction using noisy samples (see \cite{SP}) where an asymptotic estimate is obtained for the mean squared error. In particular, and as was shown in \cite{SP}, generalized Vandermonde matrices play an important role in the minimum mean squared error estimation of vector fields, as might be measured in a sensor network. In such networks, the parameter $d$ is the dimension of the field being measured, $L$ is the number of sensors and $N$ can be taken as the approximate bandwidth of the measured signal per dimension. This asymptotic can be calculated as a random eigenvalue expectation, whose limit distribution depends on the signal dimension $d$.  In the case $d=1$ the limit is via random Vandermonde matrices. As $d\rightarrow \infty$ the Marchenko--Pastur limit distribution is shown to apply. Further applications were treated in \cite{GC02} including source identification and wavelength estimation. 

\par One is typically interested in studying the behavior of these matrices as both $N$ and $L$ go to infinity at a given ratio, $\lim_{L,N\to\infty}\frac{L}{N^{d}}=\beta$. In \cite{GC02}, important results were obtained for the case $d=1$. In particular, the limit of the moments of ${\bf V^{*}}{\bf V}$ was derived and a combinatorial formula for the asymptotic moments was given under the hypothesis of continuous density. In \cite{TW}, these results were extended to more general densities and it was also proved that these moments arise as the moments of a probability measure $\mu_{\nu,\beta}$ supported on $[0,\infty)$. This measure depends on the measure $\nu$, the distribution of the phases, and on the value of $\beta$.

\par In \cite{TW}, the behavior of the maximum eigenvalue was studied and tight upper and lower bounds were found. Here we extend these results and study the maximum eigenvalue of the $d$--fold extended Vandermonde matrix. More specifically, we study the asymptotic behavior of the maximum eigenvalue of the matrix ${{\bf V}^{(d)}}^{*}{\bf V}^{(d)}$ and derive upper and lower bounds.

\par A natural question is how the smallest singular value behaves as $N\to\infty$, and this paper is one of the first to address this question. Here we restrict to the case $d=1$. The matrix $\bf{V}^{*}\bf{V}$ is an $L\times L$ positive definite random matrix with eigenvalues 
$$
0\leq \lambda_{1}(N)\leq \ldots \leq \lambda_{L}(N).
$$
The singular values of $\bf{V}$ are by definition the eigenvalues of $\sqrt{\bf{V}^{*}\bf{V}}$. Therefore, $s_{i}(N)=\sqrt{\lambda_i(N)}$. On one hand, it is clear that if $L>N$ the matrix ${\bf V^{*}}{\bf V}$ is of size $L\times L$ and rank $N$. Therefore, if $\beta>1$ the asymptotic limit measure has an atom at zero of size at least $1-1/\beta$. On the other hand, if $L=N$, the random matrix ${\bf V^{*}}{\bf V}$ has determinant
\begin{equation}
\det({\bf V^{*}}{\bf V)}=|\det({\bf V})|^{2}=\frac{1}{N^{N}}\cdot\prod_{1\leq p<q\leq N}{|e^{2\pi i\theta_{p}}-e^{2\pi i\theta_{q}}|^2}.
\end{equation}
This determinant is zero if and only if there exist distinct $p$ and $q$ such that  $\theta_{p}=\theta_{q}$. This is an event of zero probability if the probability measure has a density. Therefore, the minimum eigenvalue value $\lambda_1(N)$ is positive with probability 1 and converges to 0 as $N$ increases. In this work, we show that with high probability $\lambda_{1}(N)\leq N^2\exp(-C\sqrt{N})$. As a consequence of our argument we show that with high probability 
\begin{equation}
\max \Bigg\{\prod_{i=1}^{N}|z-z_i|^2 \,\,:\,\,|z|=1\Bigg\}\geq \exp(C\sqrt{N})
\end{equation}
where $z_k=e^{2\pi i\theta_k}$ and $\{\theta_{1},\ldots,\theta_{N}\}$ are i.i.d on $[0,1]$. Moreover, we explicitly determine the constant $C$. We believe that this may prove to be of independent mathematical interest. Additionally, we show the absence of finite moments for the matrix $({\bf V^{*}}{\bf V})^{-1}.$

\par Finally, we present a generalized version of the previously discussed random Vandermonde matrices. More specifically, consider an increasing sequence of integers $\{k_{p}\}_{p=1}^{\infty}$ and let $\{\theta_{1},\ldots,\theta_{N}\}$ be i.i.d. random variables uniformly distributed on the unit interval $[0,1]$. Let ${\bf V}$ be the $N\times N$ random matrix defined as
\begin{equation} 
V(p,q):=\frac{1}{\sqrt{N}}z_{q}^{k_p}
\end{equation} 
where $z_{q}:=e^{2\pi i\theta_{q}}$. Note that if we consider the sequence $k_p=p-1$ then the matrix ${\bf  V}$ is the usual random Vandermonde matrix defined in  (\ref{eqn_Vandermondedefn}). We study the limit eigenvalue distribution of this matrix ${\bf X}:={\bf VV}^{*}$ and in particular its asymptotic moments. We also find a combinatorial formula for its moments and show that for every sequence there exists a unique probability measure on $[0,\infty)$ with these moments. Finally, we show that for the sequence $k_p=2^p$ the limit eigenvalue distribution is the famous Marchenko--Pastur distribution. 

\par The rest of the paper proceeds as follows. In Section \ref{sec:rand_mat_ess}, we present some preliminaries in random matrix theory, the Littlewood--Offord theory (\cite{taovu}), set up some notation and terminology, and review some known results for random Vandermonde matrices. In Section \ref{tracelog}, we derive a formula for the trace log and log determinant of the random Vandemonde matrices. We also prove the absence of finite moments for the matrix ${\bf M}^{*}{\bf M}$ where ${\bf M}={\bf V}^{-1}$. In Section \ref{maxeig}, we present upper and lower bounds for the behavior of the maximum singular value of ${\bf V}^{(d)}$ in the general case. In Section \ref{sec_mineig}, we study the behavior of the  minimum singular value of ${\bf V}$. In Section \ref{sec_numer}, we present some numerical results that suggest the absence of an atom at zero for the limit eigenvalue for the square case. In the last Section, we analyze the moments and limit eigenvalue distributions of the generalized version of the random Vandermonde matrices as described before.

\section{Preliminaries}\label{sec:rand_mat_ess}

\subsection{Random Matrix Theory}

Throughout the paper we denote by ${\bf A}^{*}$ the complex conjugate transpose of the matrix ${\bf A}$ and by ${\bf I}_{N}$ the $N\times N$ identity matrix. We let $\mathrm{Tr}({\bf A}):=\sum_{i=1}^{N}{a_{ii}}$ be the non--normalized trace, where $a_{ii}$ are the diagonal elements of the matrix ${\bf A}$. We also let $\mathrm{tr}_{N}({\bf A})=\frac{1}{N}\mathrm{Tr}({\bf A})$ be the normalized trace. Let ${\bf A}_{N}=(a_{ij}(\omega))_{i,j=1}^{N}$ be a random matrix where the entries $a_{ij}$ are random variables on some probability space. We say that the random matrices ${\bf A}_{N}$ converge to a random variable $A$ in distribution if the moments of ${\bf A}_{N}$ converge to the moments of the random variable $A$, and denote this by ${\bf A}_{N}\to A$.

\par Note that for a Hermitian $N\times N$ matrix ${\bf A} = {\bf A}^{*}$, the collection of moments corresponds to a probability measure $\mu_{{\bf A}}$ on the real line, determined by $\mathrm{tr}_{N}({\bf A}^{k})=\int_{\mathbb{R}}{t^{k}\,d\mu_{A}(t)}$. This measure is given by the eigenvalue distribution of ${\bf A}$, i.e., it puts mass $\frac{1}{N}$ on each of the eigenvalues of ${\bf A}$ (counted with multiplicity):
\begin{equation}
 \mu_{{\bf A}}=\frac{1}{N}\sum_{i=1}^{N}{\delta_{\lambda_{i}}}
\end{equation}
where $\lambda_{1},\ldots,\lambda_{N}$ are the eigenvalues of ${\bf A}$. In the same way, for a random matrix ${\bf A}$, $\mu_{{\bf A}}$ is given by the averaged eigenvalue distribution of ${\bf A}$. Thus, moments of random matrices with respect to the averaged trace contain exactly the type of information in which one is usually interested when dealing with random matrices.

\par Consider an $N\times L$ random Vandermonde matrix $\bf V$ with unit complex entries, as given in  (\ref{eqn_Vandermondedefn}). The variables $\theta_{\ell}$ are called the phase distributions and $\nu$ its probability distribution. It was proved in \cite{GC02} that if $d\nu=f(x)\,dx$ for $f(x)$ continuous in $[0,1]$, then the matrices ${\bf V^{*}\bf V}$ have finite asymptotic moments. In other words, the limit
\begin{equation}
m_{r}^{(\beta)}=\lim_{N\to\infty}{\mathbb{E}\Big[\mathrm{tr}_{L}\Big(({\bf V^{*}V})^r \Big)\Big]}
\end{equation}
exists for all $r\geq 0$. Moreover, 
\begin{equation}
m_{r}^{(\beta)}=\sum_{\rho\in\mathcal{P}(r)}{K_{\rho,\nu}\beta^{|\rho|-1}}
\end{equation}
where $K_{\rho,\nu}$ are positive numbers indexed by the partition set. We call these numbers {\it Vandermonde expansion coefficients}. 

The fact that all the moments exist is not enough to guarantee the existence of a limit probability measure having these moments. However, it was proved in \cite{TW} that the eigenvalues of ${\bf V^{*}V}$ converge in distribution to a probability measure $\mu_{\beta,\nu}$ supported on $[0,\infty)$ where $\beta=\lim_{N\to\infty}{\frac{L}{N}}$. More precisely, 
$$
m_{r}^{(\beta)}=\int_{0}^{\infty}{t^{r}\,d\mu_{\beta,\nu}(t)}.
$$ 
In \cite{TW}, the class of functions for which the limit eigenvalue distribution exists was enlarged to include unbounded densities and lower bounds and upper bounds for the maximum eigenvalue were found. We suggest that the interested reader look at the articles \cite{GC02} and \cite{TW} for more properties on the Vandermonde expansion coefficients as well as methods and formulas to compute them. 

\subsection{Littlewood--Offord Theory}

Let $v_1,\ldots,v_n$ be $n$ vectors in $\mathbb{R}^d$, which we normalise to all have length at least $1$. For any given radius $\Delta > 0$, we consider the small ball probability
$$
\displaystyle p(v_1,\ldots,v_n,\Delta) := \sup_B {\mathbb{P}}( \eta_1 v_1 + \ldots + \eta_n v_n \in B )
$$
where $\eta_1,\ldots,\eta_n$ are i.i.d. Bernoulli signs (i.e. taking on values $+1$ or $-1$ independently with a probability of $1/2$), and $B$ ranges over all (closed) balls of radius $\Delta$. The Littlewood--Offord problem is to compute the quantity
$$
\displaystyle p_d(n,\Delta) := \sup_{v_1,\ldots,v_n} p(v_1,\ldots,v_n,\Delta)
$$
where $v_1,\ldots,v_n$ range over all vectors in $\mathbb{R}^d$ of length at least $1$. Informally, this number measures the extent to which a random walk of length $n$ (with all steps of size at least $1$) can concentrate into a ball of radius $\Delta$.

The one dimensional case of this problem was solved by Erd\"os. First, one observes that one can normalise all the $v_i$ to be at least $+1$ (as opposed to being at most $-1$). In the model case when $\Delta < 1$, he proved that
$$
\displaystyle p_1(n,\Delta) = \frac{1}{2^n}\binom{n}{\lfloor n/2\rfloor} = \frac{\sqrt{\frac{2}{\pi}}+o(1)}{\sqrt{n}}
$$
when $0 \leq \Delta < 1$ (the bound is attained in the extreme case $v_1=\ldots=v_n=1$). A similar argument works for higher values of $\Delta$, using Dilworth's Theorem instead of Sperner's Theorem, and gives the exact value
\begin{equation}\label{LOT}
p_1(n,\Delta) = \frac{1}{2^n}\sum_{j=1}^s \binom{n}{m_j} = \frac{s\sqrt{\frac{2}{\pi}}+o(1)}{\sqrt{n}}
\end{equation}
whenever $n \geq s$ and $s-1 \leq \Delta < s$ for some natural number $s$, where $\binom{n}{m_1},\ldots,\binom{n}{m_s}$ are the $s$ largest binomial coefficients of $\binom{n}{1}, \ldots, \binom{n}{n}$. See \cite{taovu} for more details on the Littlewood--Offord Theory.

\section{Trace Logarithm Formula and the Inverse of a Vandermonde Matrix}\label{tracelog}

\subsection{Inverse of a Vandermonde Matrix}
Given a vector $x$ in $\cmplx^N$, we define $\sigma^m_r(x)$ to be the sum of all $r$--fold products of the components of $x$ not involving the $m$--th coordinate. In other words, 
$$
\sigma^m_r(x) = \sum_{\rho_r^m} \prod_{k\in \rho_r^m} x_k
$$
where $\rho_r^m$ is a subset of  $\lc x_1,x_2,\ldots,x_{m-1},x_{m+1},\ldots,x_N \rc$ of cardinality $r$.

The following Theorem was proved in \cite{GEpaper}.
\begin{theorem} \label{GEthm}
Let ${\bf V}$ be a square $N\times N$ matrix given by
\begin{equation}
\label{eqn_Vandefn}
{\bf V} :=
\lb
\begin{array}{llll}
1      & 1        & \ldots & 1 \\
x_1 & x_2 &\ldots & x_N \\
\vdots  & \vdots       & \ddots & \vdots \\
x_1^{N-1} & x_{2}^{N-1} & \ldots & x_N^{N-1} 
\end{array}
\rb
\end{equation}
with non--zero entries. Then its inverse ${\bf M}:={\bf V}^{-1}$ is the matrix with entries
$$
\mathbf{M}(m,n) = \frac{(-1)^{N-n} \sigma_{N-n}^m(x)}{\prod_{j \neq m} \lb x_m - x_j \rb}
$$
with $m,n \in \lc 1,2,\ldots,N \rc$.
\end{theorem}

\begin{remark} 
Let $\nu_1 \leq \nu_2 \leq \ldots \leq \nu_N$ be the eigenvalues of ${\bf M}^* {\bf M}$
and let $\lambda_1 \leq \lambda_2 \leq \ldots \leq \lambda_N$ be the corresponding eigenvalues of
${\bf V}^*{\bf V}$, which are the same as for ${\bf V}{\bf V}^*$. Note that 
$$
\nu_k = \lambda^{-1}_{N-(k-1)}
$$
and in particular $\nu_N = \lambda_1^{-1}$. Therefore, to understand the behavior of $\lambda_{1}$ it is enough to understand the behavior of $\nu_{N}$. 
\end{remark}

Here we prove a Theorem about the trace log formula for random Vandermonde matrices and a Theorem about the non--existence of the moments of ${\bf M}^* {\bf M}$, but first we need the following Lemma. 

\begin{lemma}\label{lemma1}
Let ${\bf M}$ be an invertible $N\times N$ matrix with columns $X_{1},\ldots,X_{N}$ and let $R_{1},\ldots,R_{N}$ be the rows of ${\bf M}^{-1}$. Let ${\mathcal V}_i$ be the subspace generated by all the column vectors except $X_{i}$, i.e.,
$$
{\mathcal V}_i=\mathrm{span}\{X_{1},X_{2},\ldots,X_{i-1},X_{i+1},\ldots,X_{N}\}.
$$
Then the distance between the vector $X_{i}$ and the subspace ${\mathcal V}_i$ is
$$
\mathrm{dist}(X_{i},{\mathcal V}_i)=\frac{1}{\norm{R_{i}}}.
$$
Moreover,
$$
\sum_{i=1}^{N}{\lambda_{i}({\bf M}^* {\bf M})^{-1}}=\sum_{i=1}^{N}{\mathrm{dist}(X_{i},{\mathcal V}_i)^{-2}}.
$$
\end{lemma}

\begin{proof}
The result follows from an identity involving the singular values of ${\bf M}$. By definition, the inner product $\langle R_k,X_{\ell}\rangle= \delta_{k,\ell}$ so that $R_k$ is orthogonal to
${\mathcal V}_k$. Hence,
$$
d(X_k, {\mathcal V}_k) = \frac{1}{\euclidnorm{R_k}}.
$$
Let $\lambda_{k}({\bf M}^* {\bf M})$ be the eigenvalues of ${\bf M}^* {\bf M}$. Then,
$$
\tr \lb \lb {\bf M}^{-1} \rb^* {\bf M}^{-1} \rb = \sum_{i=1}^{N}{\lambda_{i}({\bf M}^* {\bf M})^{-1}}.
$$
On the other hand, 
$$
\tr \lb \lb {\bf M}^{-1} \rb^* {\bf M}^{-1} \rb = \sum_{1\leq k,\ell\leq N} \vert \lb {\bf M}^{-1} \rb_{k,\ell} \vert^2
= \sum_{k=1}^{N} \euclidnorm{R_k}^2
$$
completing the proof.
\qed \end{proof}

\begin{theorem}
Let ${\bf V}$ be a square random Vandermonde matrix of dimension $N$ with i.i.d. phases distributed according to a measure $\nu$ with continuous density $f(x)$ over $[0,1]$. Then
\begin{equation}
\mathbb{E}\Big( \mathrm{tr}_{N}\log ({\bf{V}^{*}\bf{V}})\Big) = (N-1)\,\mathbb{E}\big(\log |1-e^{2\pi i\theta}|\big)-\log(N).
\end{equation}
\end{theorem}

\begin{proof}
Let $0\leq \lambda_{1}\leq \lambda_{2}\leq \ldots \leq \lambda_{N}$ be the eigenvalues of ${\bf V}^{*}{\bf V}$. Note that $\lambda_{1}>0$ with probability one. It is clear that, 
\begin{equation}
\log\det({\bf V}^{*}{\bf V})=\sum_{i=1}^{N}{\log(\lambda_{i})}.
\end{equation}
On the other hand,
$$
\det({\bf V}^{*}{\bf V})=\frac{1}{N^{N}}\prod_{1\leq p<q\leq N}{|e^{2\pi i\theta_{p}}-e^{2\pi i\theta_{q}}|^2}
$$
hence
\begin{equation}
\log\det({\bf V}^{*}{\bf V})=\sum_{p<q}{2\log(|e^{2\pi i\theta_{p}}-e^{2\pi i\theta_{q}}|)}-N\log(N).
\end{equation}
Since the phases are identically distributed, it is easy to see that the expectation $\mathbb{E}\big[\log\det({\bf V}^{*}{\bf V})\big]$
is
\begin{equation}\label{eq1}
\frac{N(N-1)}{2}\,\,\mathbb{E}\Big(2\log |e^{2\pi i\theta_{p}}-e^{2\pi i\theta_{q}}| \Big)-N\log(N) 
\end{equation}
which is
\begin{equation}
N(N-1)\,\mathbb{E}\Big(\log |1-e^{2\pi i\theta}| \Big)-N\log(N).
\end{equation}
Since for every invertible Hermitian matrix ${\bf A}$, we have that $\mathrm{Tr}\log({\bf A})=\log\det({\bf A})$, in particular we have that
\begin{equation}\label{eq2}
\mathrm{tr}_{N}\log({\bf V}^{*}{\bf V})=\frac{1}{N}\log\det({\bf V}^{*}{\bf V}).
\end{equation}
Combining (\ref{eq1}) and (\ref{eq2}) we see that 
\begin{equation}
\mathbb{E}\Big(\mathrm{tr}_{N}\log({\bf V}^{*}{\bf V})\Big)=(N-1)\,\mathbb{E}\big(\log |1-e^{2\pi i\theta}| \big)-\log(N).
\end{equation}
\qed \end{proof}

\begin{theorem}\label{momentsinv}
Let ${\bf V}$ be a square $N\times N$ random Vandermonde matrix. Then for every $N\geq 2$, the matrix ${\bf V^{*}}{\bf V}$ is invertible with probability $1$ and 
$$
\mathbb{E}\big(\mathrm{tr}_{N}\big(({\bf V^{*}}{\bf V})^{-p}\big)\big)=\infty
$$
for every $p\geq 1$.
\end{theorem}

\begin{proof}  
It is enough to prove the case $p=1$ since the other cases follow from this one. As we mentioned in the introduction, the matrix ${\bf V^{*}}{\bf V}$ is invertible with probability one since its determinant its non--zero with probability one. From Lemma \ref{lemma1} we see that 
$$
\mathbb{E}\big(\mathrm{tr}_{N}\big(({\bf V^{*}}{\bf V})^{-1}\big)\big) = \frac{1}{N}\sum_{m,n=1}^{N}{\mathbb{E}\big( |\mathbf{M}(m,n)|^2\big)}
$$
where $\mathbf{M}=\mathbf{V}^{-1}$. In particular, 
$$
\mathbb{E}\big(\mathrm{tr}_{N}\big(({\bf V^{*}}{\bf V})^{-1}\big)\big) \geq \frac{1}{N}\mathbb{E}\big(|\mathbf{M}(1,N)|^2\big).
$$
Now using Theorem \ref{GEthm} we see that 
$$
\mathbb{E}\big(|\mathbf{M}(1,N)|^2\big)=\Bigg( \frac{1}{2\pi}\int_{0}^{2\pi}{\frac{d\theta}{|1-e^{i\theta}|^2}}\Bigg)^{N-1}=\infty.
$$
\qed \end{proof}

\section{Maximum Eigenvalue}\label{maxeig}

Let ${\bf X}$ be the $L\times L$ matrix defined as ${\bf X} := {{\bf V}^{(2)}}^{*} {\bf V}^{(2)}$. Suppose that the 
phases $(\theta_k,\psi_k)$ are selected i.i.d. on $[0,1]^2$ with density $f(x)$. We further suppose, as in \cite{TW}, that $(\theta_k,\psi_k) - (\theta_1,\psi_1)$ given $\theta_1,\psi_1$ has a conditional density (which exists for all  $(\theta_1,\psi_1)$ if it exists for one) and denote this density by 
$f$ ignoring its dependence on $(\theta_1,\psi_1)$, as it only appears through $\maxnorm{f}$. It can be shown that 
${\bf X}$ has the same eigenvalues as the matrix ${\bf A}$ whose entries are 
\begin{equation}\label{eqn_Amat}
A(k,m) := D_N\Big(2\pi(\theta_k - \theta_m)\Big) D_N\Big(2\pi(\psi_k - \psi_m)\Big)
\end{equation}
where 
$$
D_N(x) := %\Bigg( 
            \frac{\sin(\frac{N}{2}x)}{N\sin(\frac{x}{2})} %\Bigg)
$$
is the Dirichlet kernel (see e.g. \cite{TW}). Similarly, in the case $d\geq 3$ we obtain a product of $d$ Dirichlet kernels. Subsequently, ${\bf A}$ is used to construct upper and lower bounds for the maximum eigenvalue
$\lambda_L$. 

\par We now proceed to obtain asymptotic upper and lower bounds for $d$--fold Vandermonde matrices. We first focus on the case $d=2$ and retain the notation of Section \ref{sec:rand_mat_ess}. In what follows we
prove the following Theorem.
\begin{theorem}
Let ${\bf V}^{(d)}$ be the $N^{d}\times L$ $d$--fold random Vandermonde matrix defined in (\ref{eqn_vandergendef}). Let $\lambda_L$ be the maximum eigenvalue of the matrix $({\bf V}^{(d)})^{*} ({\bf V}^{(d)})$. If $\lim_{N,L\to\infty}\frac{L}{N^{d}}=\beta \in (0,\infty)$,
there exists constants $C_e$ such that for all $C > C_e$
\begin{equation}\label{eqn_maxupperbnd}
\mathbb{P} \Big(\lambda_L \geq C\log(N^d) + u \Big) \leq e^{-u} \frac{N^d}{N^{\delta (d-1)\log(N)}}
\end{equation}
\end{theorem}
where $\delta = C - C_e$.

\begin{proof}
The line of argument follows that of \cite{TW}.  We begin with the following upper bound for the Dirichlet function proved in \cite{TW},
\begin{equation}
\label{eq_char} 
\babs{\frac{\sin(\frac{N}{2}\abs{x})}{N\sin(\frac{\abs{x}}{2})}}\leq\sum_{k=1}^{N/2}{\,\frac{1}{k}\,\mathbf{1}_{\big[\frac{2\pi(k-1)}{N},\frac{2\pi k}{N}\big)}(\abs{x})}
\end{equation}
where $\mathbf{1}_{B}$ is the indicator function on the Borel set $B$. To apply the bound, let $p_{a,b}$, with $a,b\in\mathbb{Z}$, be the probability that  
$$
\theta_k - \theta_1\in [(a-1)/N,a/N]
$$
and
$$
\psi_k - \psi_1\in [(b-1)/N,b/N].
$$
Define,
$$
q_{a,b} := p_{a,b} + p_{-a,b} + p_{a,-b} + p_{-a,-b}
$$
Then, it is easy to see from the union bound that
\begin{equation}
q_{a,b}\leq \frac{4}{N^{2}} \norm{f}_\infty = \frac{C_q}{N^2}
\label{eqn_qbnd}
\end{equation}
where the function $f$ is the probability density over the unit square $[0,1]^{2}$ and $C_{q}$ a constant. Next for the magnitude of a term in the first row, corresponding to $\theta_1,\psi_1$, we find that
\begin{equation}
\mathbb{E} \big(\exp(Xt)\mid \theta_1, \psi_1\big) \leq \sum_{a=1}^{N/2} \sum_{b=1}^{N/2} q_{a,b} e^{\frac{t}{ab}} \leq 1 + (e-1) \sum_{a=1}^{N/2} \sum_{b=1}^{N/2}  q_{a,b} \frac{t}{ab}
\end{equation}
where $X$ is defined as
$$
X := \Big| D_N(2\pi(\theta_k - \theta_1)) D_N(2\pi(\psi_k - \psi_1)) \Big|
$$
for $k\neq 1$ and the upper bound does not depend on $k$ or $(\theta_1,\psi_1)$.
%and $q_{a,b} := p_{a,b} + p_{a,-b} + p_{-a,b} + p_{-a,-b}$ where the definition of the latter is clear. 
%It follows that
%$$
%q_{a,b} \leq \frac{16 \pi^2}{(2M+1)^2} \norm{f}_\infty := \frac{C_q}{(2M+1)^2}
%$$
If $R$ is any row sum of the entries of the matrix ${\bf X}$ it follows that
\begin{equation}
\mathbb{E}\big(\exp(R)\big) \leq e \Bigg( 1 + C_q \frac{(e-1)}{N^2} H^2_{N/2} \Bigg)^{L-1} \leq e^{\beta_N C_q (e-1) H^2_{N/2}}
\label{eqn_Rbnd}
\end{equation}
by taking $t=1$ and using the facts that $\lb 1 + x/y \rb^y \leq \exp(x)$ and $\beta_N \rightarrow \beta$ as $N \rightarrow \infty$, where $H_p:=\sum_{k=1}^p 1/k$.  
Then
$$
H_N = \log N + \gamma + \frac{1}{2N} + O(N^{-2})
$$
where $\gamma$ is the Euler--Mascheroni constant. It follows that 
\begin{equation}
\mathbb{E} \big( \exp(R)\big) \leq e^{C_e \lb \log N \rb^2} = N^{C_e \log N}
\end{equation}
for some suitable constant $C_e > 0$. Applying the union bound to the maximum row sum $Y$ and then, Markov's inequality with $C \geq C_e + \delta$, we observe that 
\begin{equation}
\mathbb{P}\Big(Y \geq C \log(N^{2}) + u\Big) \leq e^{-u} \frac{N^2}{N^{C\log N}} N^{C_e \log N} = e^{-u} \frac{N^2}{N^{\delta\log N}}. 
\label{eqn_unionbnd}
\end{equation}
Since the maximum eigenvalue is upper bounded by the maximum row sum of magnitudes, then $Y \geq \lambda_L$. This concludes 
the proof for the case $d=2$.

\par For the case $d > 2$, one obtains a $d$--fold product of Dirichlet functions such that the exponent of the harmonic function $H$ is $d$ instead of 2 in (\ref{eqn_Rbnd}) and also in (\ref{eqn_qbnd}). Constant terms are also suitably modified and with these changes carried through to (\ref{eqn_unionbnd}) where 2 is again replaced with $d$, the result follows as before. 
\qed \end{proof}
The following Corollary is stated without proof.
\begin{cor}
There exists a positive constant $B$ such that
$$
\mathbb{E}(\lambda_L) \leq B \log (N^d) + o(1).
$$
\end{cor}

\subsection{Lower Bound}

The purpose of this subsection is to present the following lower bound for the maximum eigenvalue $\lambda_L$,
in which we suppose that phases are provided according to a joint continuous density $f$ bounded away from zero. 

\begin{theorem}
Let ${\bf V}^{(d)}$ be the $N^{d}\times L$ $d$--fold random Vandermonde matrix and let $\lambda_L$ be its maximum eigenvalue. If $N,L \rightarrow \infty$ such that $\lim_{L,N\to\infty}\frac{L}{N^{d}}=\beta\in (0,\infty)$,  
there exists a constant $K$ such that 
\begin{equation}
\mathbb{P}\Big( \lambda_L \geq K\frac{ \log N^d}{ \log \log N^{d}}\Big) = 1 - o(1).
\label{eqn_maxlowbnd}
\end{equation}
\end{theorem}

\begin{proof}
We rely on the equivalent matrix given in (\ref{eqn_Amat}), which is an $L \times L$ matrix. For notation simplicity, we specialize to the case $d=2$ since the case $d\geq 3$ follows a similar argument. To construct our lower bound by analogy with the $1$--dimensional case, we want to obtain large numbers of points $(\theta, \psi)$ that lie close together, that is points such that $\vert \theta_k - \theta_m \vert$ and $\vert \psi_k - \psi_m \vert$ are both small.  If this is the case for some large set of indexes $\Gamma$, then it follows that $A(k,m) \approx 1$ for all $k$ and $m$ in $\Gamma$. In addition, the matrix ${\bf A}$ is symmetric with diagonal values 1 and so it follows that the eigenvalues of ${\bf A}$ interlace with the eigenvalues of any principal sub--matrix (see \cite{Wilkinson}). In particular, if we define the matrix ${\bf A}_\Gamma$ as
$$
{\bf A}_\Gamma(k,m):= A(k,m)
$$
for $k,m\in \Gamma$ then $\lambda_1({\bf A}_\Gamma) \leq \lambda_1({\bf A})$. Hence, if 
$\abs{\theta_k - \theta_m}$ and $\abs{\psi_k-\psi_m} < \epsilon$ for some $\epsilon>0$ for all $k$ and $m$ then it follows that 
$$
D_N^2(2 \pi \epsilon) \times \vert \Gamma \vert \leq \lambda_1({\bf A}_\Gamma) \leq \lambda_1({\bf A}).
$$
Divide the unit square $[0,1] \times [0,1]$ into $N^2_\epsilon$ equal squares with sides of length $\epsilon$. Take $\Gamma$ to be the indexes corresponding to the square with the most number of points in it. By hypothesis, the joint measure has a continuous density $f$ bounded away from 0 throughout $[0,1]^2$. Therefore, it follows that 
$$ 
\eta:=\min \,\{f(\theta,\psi)\,:\,(\theta,\psi)\in [0,1]^2\}>0.
$$
By construction, we select $L$ phase points $(\theta,\psi) \in [0,1]\times[0,1]$ so that each
square receives at least $\epsilon^2 L N^{-2} \eta$ points on the average, which is $O(1)$ since $L/N^2 \rightarrow \beta$ as $L,N \rightarrow \infty$. This is an occupancy model and we are interested in the square with the maximum number of points. The number of such points is at least
\begin{eqnarray*}
k(N^2) & = & \alpha \frac{ \log N_\epsilon^2}{ \log \log N_\epsilon^2} \\
       & \geq & \alpha (1 - o(1))\frac{\log N}{\log \log N}
\end{eqnarray*}
with probability $1- o(1)$ for any $\alpha \in (0,1)$ independently of the mean number of points
per square, (see for instance (see e.g. \cite{Raab}). However since $\epsilon > 0$ and $\alpha \in (0,1)$ are both arbitrary, the lower bound on $\lambda_1({\bf A}_\Gamma)$ and hence on $\lambda_1({\bf A})$ follows with $K$ any constant in $(0,1)$. The proof is complete.
\qed 
\end{proof}

\section{Minimum Eigenvalue}\label{sec_mineig}

In this Section we focus on the behavior of the minimum eigenvalue $\lambda_{1}$ for the case $d=1$. Consider the matrix ${\bf A}$ as in  (\ref{eqn_Amat}) and all its $2\times 2$ principal sub--minors. These matrices are symmetric and the minimum eigenvalue
for the sub--minor determined by phases $\theta_k$ and $\theta_\ell$ is denoted $\lambda_{k,\ell}$. In other words, $\lambda_{k,\ell}$ is the smaller root of the equation
$$
\lb 1 - \lambda_{k,\ell} \rb^2 = D_N(2\pi(\theta_k - \theta_\ell))^2.
$$
Taking square roots and applying again the interlacing Theorem (see \cite{Wilkinson} for a reference) we obtain, 
$$
\lambda_1(N) \leq \min_{k,\ell} \Big( 1 - D_N(2\pi(\theta_k-\theta_\ell)) \Big).
$$
Let $\omega$ and $\alpha$ be defined as $\omega := N^2 \min_{k,\ell} \vert \theta_k - \theta_\ell \vert$ and
$$
\alpha := \min_{k,\ell}\Big(1 -D_N(2\pi(\theta_k-\theta_\ell))\Big).
$$
From a result of de Finetti (see \cite{Feller_Vol2} for a reference), 
$$ 
\mathbb{P} \Big(\min_{k,\ell} \vert \theta_k - \theta_\ell \vert > \delta\, \Big) = \lb 1 - N \delta \rb_+^{N-1}
$$ 
where $ \lb x \rb_+ := \max \lc x , 0 \rc$. Substituting $\eta/N^2$ for $\delta$ and taking the limit as $N\to\infty$, we obtain that
$$
\lim_{N\to\infty} \mathbb{P} \Big( \min_{k,\ell} \vert \theta_k - \theta_\ell \vert > \frac{\eta}{N^2} \,\Big)=e^{-\eta}.
$$
On taking the Taylor expansion of $D_N$ to second order around the origin we obtain
$$
\min_{k,\ell} \Big(1 - D_N(2\pi(\theta_k-\theta_\ell))\Big) = \frac{\pi^2\omega^2}{6 N^2} + o(N^{-2}).
$$
Therefore, the following limit holds,
$$ 
\lim_{N\to\infty}\mathbb{P}\Big( \alpha \leq \frac{\pi^2\omega^2}{6 N^2} \Big) = 1 - e^{-\omega}.
$$
Then, it follows that $\lambda_1(N) \leq O(N^{-2})$ as $N \rightarrow \infty$. As we now show, the approach to zero is much 
more rapid. To obtain better estimates for $\lambda_1(N)$ as $N \rightarrow \infty$, we now consider the maximum eigenvalue of the inverse matrix. 

Given an $N\times N$ matrix ${\bf M}$, we have the following matrix norms $\spnorm{{\bf M}}:= \sup_{j} \big( \sum_{k=1}^N \vert M_{k,j} \vert \big)$ and $\opnorm{{\bf M}}  :=  \sqrt{\lambda_N({\bf M}^* {\bf M})}$. The following inequality is well known (see \cite{HornJohnson} for more details),
$$
\frac{1}{\sqrt{N}} \spnorm{{\bf M}} \leq \opnorm{{\bf M}} \leq \spnorm{{\bf M}}. 
$$
We now prove the following Lemma which is used later.
\begin{lemma}\label{lemma_poly}
Let $A(z) = a_0 z^n + a_1 z^{n-1} + \ldots + a_n$ be a complex polynomial and
let $A^* := \max_{\lvert z \rvert = 1} \lvert A(z) \rvert$ be its maximum on the unit circle. Then

\begin{equation}
 \frac{\lvert a_n \rvert + \ldots + \lvert a_0 \rvert}{n+1} \leq A^* \leq  \lvert a_n \rvert + \ldots + \lvert a_0 \rvert.
\label{eqn_sstar}
\end{equation}

\end{lemma}
\begin{proof}
The second inequality follows immediately from the triangle inequality, so we concentrate on the first one.
It is enough to show that
\begin{equation}
\lvert a_k \rvert \leq A^*
\label{eqn_aineq}
\end{equation}
for all $k$. By applying Cauchy's integral Theorem and using the fact that
$$
\int_{\lvert z \rvert = 1}{z^{-i}\,dz} = 0
$$
for all $i\neq 1$, we obtain that,
\begin{equation}
a_k = \frac{1}{2\pi i} \int_{\lvert z \rvert = 1} \frac{A(z)}{z^{n-k+1}} dz
\end{equation}
for all $k$. Therefore,
\begin{equation}
\lvert a_k \rvert \leq \frac{1}{2\pi} \int_{0}^{2\pi} \lvert A(z(\theta)) \rvert d\theta \leq A^* \nonumber 
\label{eqn_polyineq}
\end{equation}
where the first inequality follows by upper bounding the integral, taken as a line integral with respect to $\theta$ around the unit circle, and the second inequality follows from the definition of $A^*$. By applying the inequality for each $k$ we obtain the required lower bound.
\qed \end{proof}

In the following steps we find a bound on $\lambda_1(N)$ in terms of the maximum of a 
polynomial with roots on the unit circle. We begin with some definitions. Let 
$z_k = e^{2\pi i \theta_k}$ be the values determining the random Vandermonde matrix as in  (\ref{eqn_Vandermondedefn}). Let $P(z)$ be the polynomial defined as
$$
P(z) := \prod_{k=1}^N \lb z - z_k \rb.
$$
We further denote 
$$
P_p(z) := \frac{P(z)}{|z - z_p|}.
$$
Let ${\bf M}={\bf V}^{-1}$ be the inverse of the random Vandermonde matrix and let $M(p,q)$ denote its entries. Define
\begin{equation}
\beta_p := \sum_{q=1}^{N}{|M(p,q)|}.
\end{equation}
By Theorem \ref{GEthm}, we know that
\begin{equation}
\label{eqn_dstar}
\beta_p := \frac{\sqrt{N}}{\prod_{q\neq p}{|z_p-z_q|}} \Big(|\sigma_0^p| + \ldots + |\sigma^p_{N-1}| \Big).
\end{equation}
In addition, let
\begin{equation}
%| T_p (z)|:= \frac{|P(z)|}{|P(z_p)|} = \prod_{q\neq p}{\frac{| z - z_q |}{|z_p - z_q|}}.
T_p (z):= \frac{P(z)}{|P(z_p)|} = \prod_{q\neq p}{\frac{\lb z - z_q \rb }{|z_p - z_q|}}.
\label{eqn_Tpdefn}
\end{equation}
It follows from (\ref{eqn_dstar}) and (\ref{eqn_sstar}) that
$$
\frac{\beta_p}{N} \leq \max_{|z|=1} \Big( \sqrt{N} |T_p(z)| \Big) \leq \beta_p
$$
for all $p=1,\ldots,N$. The following Lemma is a direct consequence of the Hadamard's inequality (see \cite{HornJohnson}).

\begin{lemma}
\label{lem_hadbnd}
Let $z_1,\ldots,z_N$ be distinct points on the unit complex circle. Then
there exists $p_0 \in \lc 1, \ldots, N \rc$ such that,
\begin{equation}
\prod_{q\neq p_0} \lvert z_{p_0} - z_q \rvert \leq N.
\end{equation}
\end{lemma}

\begin{proof}
Assume this is not true. Therefore, for every $p$
\begin{equation}
\prod_{q\neq p} \lvert z_{p} - z_q \rvert > N.
\end{equation}
Hence, $\sum_{q\neq p}{\log|z_p-z_q|}> \log N$ for every $p$. Let ${\bf G}$ be the Vandermonde matrix whose entries are $G(i,j)=z_{j}^{i-1}$. Then $|\det({\bf G})|=\prod_{1\leq p <q \leq N}{|z_p-z_q|}$ and
$$
\log|\det({\bf G})|=\sum_{1\leq p<q\leq N}{\log|z_p-z_q|}=\frac{1}{2}\sum_{p=1}^{N}\sum_{q\neq p}{\log|z_p-z_q|}>\frac{N\log N}{2}. 
$$
Thus, $|\det({\bf G})|> N^{N/2}$ which violates Hadamard's inequality.
\qed \end{proof}

We are now in a position to prove the following Lemma, which provides upper and lower bounds on the
minimum eigenvalue of a Vandermonde matrix in terms of the polynomial $T_p$ defined in (\ref{eqn_Tpdefn}).

\begin{lemma}\label{lem_polyineq}
Let $\lambda_1(N)$ be the minimum eigenvalue of the random matrix ${\bf V}^*{\bf V}$ and $T_p(z)$ be as defined above. Then,
\begin{equation}
\frac{1}{N^3 \max_p \lc \max_{\vert z \vert = 1} \vert T_p(z) \vert^2 \rc}\leq  \lambda_1(N)
\leq \frac{1}{\max_p \lc \max_{\vert z \vert =1} \lvert T_p(z) \rvert^2 \rc}.
\end{equation}
Moreover,
\begin{equation}
\lambda_1(N) \leq \frac{N^2}{\max_{\lvert z \rvert = 1} \prod_{q\neq p_0} \lvert z - z_q \rvert^2} \leq \frac{4N^2}{\max_{\lvert z \rvert = 1} \prod_{q} \lvert z - z_q \rvert^2}.
\label{eqn_upperbnd}
\end{equation}
\end{lemma}

\begin{proof}
Since $\spnorm{{\bf M}} = \max \lc \beta_p : p=1,\ldots,N \rc$, it follows that
$$
\frac{\spnorm{{\bf M}}}{N} \leq \sqrt{N} \max_{p} \lc \max_{|z|= 1} \,|T_p(z)|\rc \leq \spnorm{{\bf M}}.
$$
On the other hand,
$$
\frac{1}{\spnorm{{\bf M}}^2} \leq \lambda_1(N) \leq \frac{N}{\spnorm{{\bf M}}^2}
$$
from which we can deduce that
\begin{equation}
\frac{1}{N^3 \max_p \lc \max_{|z|= 1} |T_p(z)|^2 \rc} \leq \lambda_1(N) \leq\frac{1}{\max_p \lc \max_{\vert z \vert =1} \lvert T_p(z) \rvert^2 \rc}.
\end{equation}
Using Lemma \ref{lem_hadbnd}, we know that there exists $p_0$ such that $\prod_{q \neq p_0} \lvert z_{p_0} - z_q \rvert \leq N$. We thus obtain that
\begin{equation}
\max_{\lvert z \rvert = 1} \,|T_{p_0}(z)| \geq \max_{|z|=1} \lb \frac{1}{N} \prod_{q \neq p_0} \lvert z - z_q \rvert \rb.
\label{eqn_Tineq}
\end{equation}
Therefore, using the fact that
$$
\max_{p} \max_{|z|=1} \,|T_{p}(z)| \geq \max_{\lvert z \rvert = 1} \,|T_{p_0}(z)|
$$
and Lemma \ref{lem_polyineq}, we see that
\begin{equation}
\lambda_1(N) \leq \frac{N^2}{\max_{\lvert z \rvert = 1} \prod_{q\neq p_0} \lvert z - z_q \rvert^2} \leq \frac{4N^2}{\max_{\lvert z \rvert = 1} \prod_{q} \lvert z - z_q \rvert^2}
\end{equation}
where the last inequality follows from the fact that $|z-z_{p_0}|\leq 2$.
\qed \end{proof}

\subsection{Stochastic Construction}

Before stating our upper bound for the minimum eigenvalue, we introduce the following definitions.
First, we define a random sequence via a realization of the Brownian bridge $\Wc$ on $[0,2\pi]$, which satisfies
$\Wc(0) = \Wc(2\pi) = 0$ (see \cite{Billingsley} for details on the Brownian bridge construction). A shift $\varphi$ of the Brownian bridge is defined by

\begin{equation*}
\Wc_\varphi(\theta) :=
\begin{cases}
\Wc(\varphi+\theta) - \Wc(\varphi), & \text{if \,\,\,} 0\leq \theta\leq 2\pi -\varphi,\\
\Wc(\varphi+\theta-2\pi) - \Wc(\varphi), & \text{if \,\,\,} 2\pi - \varphi \leq \theta < 2 \pi. 
\end{cases}
\end{equation*}

Further, define the infinite sequence $\Phi:=\{\varphi_r\}_{r\geq 0}$ to be the sequence of dyadic phases on $[0,2\pi]$. Given a realization of the Brownian bridge, define the following function,
$$
I_\varphi :=  \int_0^{2\pi} \Wc_\varphi(\theta) \frac{\sin \theta}{1 - \cos \theta} d \theta
$$
for $\varphi\in[0,2\pi]$. Note that it is not clear that the above integral is well defined, since it may not exist as the fraction
$ \frac{\sin \theta}{1 - \cos \theta}$ behaves like $\theta^{-1}$ near $0$ and $2 \pi$. We address this
matter shortly. In Figure \ref{fig_brownrlse}, we show a realization of the Brownian bridge and a shifted version with $\varphi=3\pi/2$. In Figure \ref{fig_integ}, we show 
$I_\varphi$ for the same realization.

\begin{figure}[!Ht]
  \begin{center}
    \includegraphics[width=10cm]{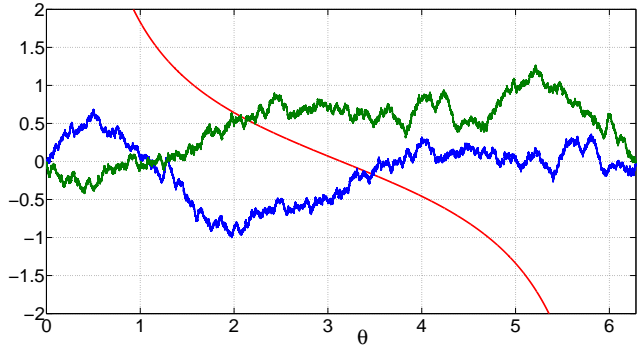}
    \caption{Realization of the Brownian bridge and a shifted version. In red we have the plot of the function $\frac{\sin\theta}{1-\cos\theta}$, in blue $\Wc$ and in green $\Wc_{\varphi}$ with $\varphi=\frac{3\pi}{2}$.}
    \label{fig_brownrlse}
  \end{center}
\end{figure}

\begin{figure}[!Ht]
  \begin{center}
    \includegraphics[width=10cm]{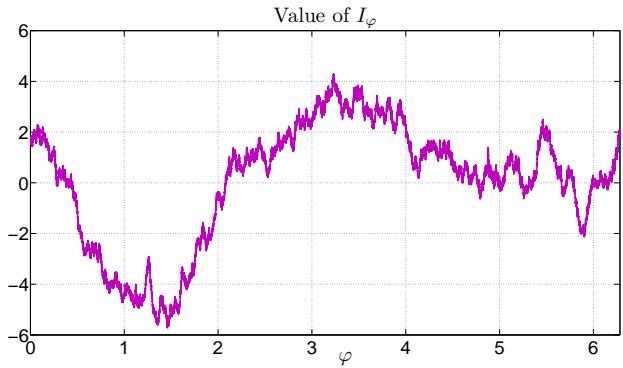}
    \caption{$I_\varphi$ for the previous Brownian bridge realization.}
    \label{fig_integ}
  \end{center}
\end{figure}

Using the sequence $\Phi$ and the same realization of $\Wc$ we construct a sequence of random variables ${\bf I}=\{I_{r}\}_{r\geq 0}$ as $I_r :=  I_{\varphi_r}$. We show that $I_\varphi$ is continuous on the interval $[0,2\pi]$, and so there exists a value $\varphi^*$ that determines the maximum value of  $I_\varphi$, which we denote as $I^*$. Since $\Phi$ is dense on the unit circle, it 
follows that  
\begin{equation}\label{eqqmax}
I^{*} := \sup \lc  I_r : r \in \NN \rc
\end{equation}
and its distribution is determined via the infinite sequence $I_r$. We now show that the above random function is well defined and that integrals are a.s. finite. 
\begin{lemma}
Given a realization of the Brownian bridge $\Wc$, then a.s. the following integrals exist for all $\varphi \in [0, 2\pi)$
$$
|I_\varphi|=\Bigg|\int_0^{2\pi} \Wc_\varphi \frac{\sin \psi}{1 - \cos \psi} d \psi \Bigg| < \infty.
$$
In addition, the function $\varphi\mapsto I_\varphi$ is continuous. 
\end{lemma}

\begin{proof}
When $\varphi=0$ we write the above integral as $I$. The Levy global modulus of continuity tells us that standard Brownian motion $B$ on $[0,2\pi)$ satisfies almost surely 
$$
 \lim_{\delta\to 0} \limsup_{0 \leq t \leq 2 \pi - \delta} \frac{\abs{B(t+\delta) - B(t)}}{w(\delta)} = 1
$$
where $w(\delta) = \sqrt{2 \delta \log \frac{1}{\delta}}$ (see \cite{RogersWilliams} for a proof of this fact). Since $\Wc$ is by definition,
$$
\Wc(\psi) = B(\psi) - \frac{\psi}{2 \pi} B(2 \pi)
$$
our argument is the same no matter which value of $\varphi$ is chosen because the Levy modulus applies to the entire sample path. We therefore set $\varphi = 0$. By definition of the Levy modulus,
almost surely there exists $\delta_2 > 0$ such that
$$
\frac{\abs{B(t+\delta) - B(t)}}{w(\delta)} \leq 2
$$
for all $0<\delta\leq\delta_2$.
Therefore,
\begin{equation}
a(\delta):=\abs{\Wc(\psi+\delta) - \Wc(\psi)} \leq 2 w(\delta) + \frac{|B(2\pi)|}{2\pi} \delta.
\end{equation}
We may therefore split the integral as,
\begin{eqnarray*}
I & = & \int_{\delta_2}^{2\pi- \delta_2} \Wc(\psi) \frac{\sin \psi}{1 - \cos \psi} d \psi + \int_0^{\delta_2} \Wc(\psi)\frac{\sin \psi}{1 - \cos \psi} d \psi \\
& + & \int_{2 \pi - \delta_2}^{2\pi} \Wc(\psi) \frac{\sin \psi}{1 - \cos \psi} d \psi.
\end{eqnarray*}
The first integral is finite, being the integral of a continuous function over the interval $[\delta_2, 2\pi - \delta_2]$. We may further suppose that $\delta_2$ has been chosen such that 
$\abs{\psi \frac{\sin \psi}{1 - \cos \psi}} \leq 4$ for $0 < \psi < \delta_2$ with a corresponding inequality in a similar neighbourhood of $2 \pi$. By choice of $\delta_2$, we obtain that  
$$
\Big| \int_0^{\delta_2} \Wc(\psi) \frac{\sin \psi}{1 - \cos \psi} d \psi \Big|
\leq  8 \int_0^{\delta_2} \frac{a(\psi)}{\psi} d \psi < O(\delta_2^{1/3})
$$ 
for sufficiently small $\delta_2$. The same argument applies to the last integral. Since $w(\delta_2)$ gives
a uniform bound the result holds for all $\varphi \in [0, 2 \pi)$. Continuity in $\varphi$ follows by a similar argument, 
\begin{eqnarray}
\vert I_\varphi - I_{\tdvarphi} \vert & \leq & \Big| \int_\delta^{2\pi- \delta} \lb \Wc_\varphi - \Wc_{\tdvarphi} \rb
\frac{\sin \psi}{1 - \cos \psi} d \psi \Big| \\
& + & \int_0^{\delta} \abs{\Wc_\varphi(\psi)} \Big|\frac{\sin \psi}{1 - \cos \psi}\Big| d \psi \nonumber \\
 & + & \int_{2\pi - \delta}^{2\pi} \abs{\Wc_\varphi(\psi)} \Big|\frac{\sin \psi}{1 - \cos \psi}\Big| d\psi \nonumber \\
 & + & \int_0^{\delta} \abs{\Wc_{\tdvarphi}(\psi)} \Big|\frac{\sin \psi}{1 - \cos \psi}\Big|  d\psi \nonumber \\
 & + &  \int_{2\pi - \delta}^{2\pi} \abs{\Wc_{\tdvarphi}(\psi)} \Big|\frac{\sin \psi}{1 - \cos \psi}\Big| d\psi. \nonumber 
\end{eqnarray}
Provided that $0 < \delta < \delta_2$, the tail integrals are all $O(\delta^{1/3})$ as before. We bound the first integral by two positive integrals, to obtain
\begin{eqnarray*}
\Big|\int_\delta^{2\pi- \delta} \lb \Wc_\varphi - \Wc_{\tdvarphi} \rb \frac{\sin \psi}{1 - \cos \psi} d \psi \Big|
%& \leq & 2 \sup \,\abs{\Wc_{\varphi}(\psi) - \Wc_{\tdvarphi}(\psi)} \int_\delta^\pi \frac{\sin \psi}{\lb 1 - \cos \psi \rb} d\psi \\ 
& \leq &  2 \sup \,\abs{\Wc_\varphi(\psi) - \Wc_{\tdvarphi}(\psi)} \big[ \log(1 - \cos\psi) \big]_\delta^\pi \\
& \leq & 6\,a(\delta) \lb \log 2 - \log (1 - \cos \delta) \rb
\end{eqnarray*}
provided $\abs{\varphi - \tdvarphi} < \delta$. Finally, $a(\delta) \lb \log 2 - \log (1 - \cos \delta) \rb \to 0$ as $\delta \to 0$, which implies continuity.
\qed \end{proof}

It therefore follows that $I^*$ is well defined. Let $T_{N}(\varphi)$ be defined as 
\begin{equation}
T_{N}(\varphi) := \frac{1}{\sqrt{N}}\log|P(e^{i\varphi})|^2 =  \frac{1}{\sqrt{N}} \sum_{q=1}^N \log \Big(2(1 - \cos(\varphi-\theta_q))\Big)
\label{eqn_TNdefn}
\end{equation}
where $P(z)$ is a random polynomial on the unit circle with roots $\{e^{i\theta_{q}}\}_{q=1}^{N}$ as before. Furthermore, define the infinite sequence of random variables $T_N(\varphi_r)$ by evaluating the previous expression at the phases of $\Phi$. Note that $T_N$ cannot be defined as a random function in either $C[0,2\pi]$ or in $D[0,2\pi]$, as its discontinuities are not of the first kind. We remark that since there are only countably many $\varphi_r$ and the phases $\theta_q$ are i.i.d. and uniformly distributed, no $\varphi_r$ coincides with any $\theta_q$ almost surely so that the sum exists.  We further observe that since, 
$$
\int_0^{2 \pi}{\log \big(2(1 - \cos\psi)\big)\,d\psi} = 0
$$
and 
$$
\int_0^{2\pi}{\log^2\big(2(1 - \cos\psi)\big)\,d\psi} < \infty.
$$
Note that the sequence $T_N(\varphi_r)$ satisfies the central limit Theorem as a function of $N$ for every $r \in \ZZ$. We consider the sequence ${\bf T}_N:=\{T_N(\varphi_r)\}_{r\geq 0}$ in the sequence space ${\mathbb R}^\infty$ with metric
$$
\rho_0({\bf x}, {\bf y}) = \sum_{\ell=0}^\infty \frac{ \abs{x_\ell - y_\ell}}{1 +  \abs{x_\ell - y_\ell}} 2^{-\ell}
$$
and using the ordering stated earlier. It is well known that this forms a Polish space (\cite{Billingsley68}). We now derive one more Lemma for use later on.
\begin{lemma}
\label{lem_Dctsmap}
Let $Y$ be a function in $D[0,2\pi]$. Then $Y$ is Lebesgue measurable, and its integral exists,
\begin{equation}
\int_0^{2\pi} Y(s) ds < \infty.
\end{equation}
Furthermore, let $Y_n$ be a sequence of functions in $D[0,2\pi]$ such that 
$Y_n \rightarrow Y$ in $D$ (i.e., with respect to the Skorohod topology). Then
\begin{equation}
\int_0^{2\pi} Y_n(s) ds \rightarrow \int_0^{2\pi} Y(s)ds.
\end{equation}
\end{lemma}

\begin{proof}
The existence of the integral follows from Lemma 1, page 110 of \cite{Billingsley68} and the subsequent discussion
which shows that functions in $D$ on a closed bounded interval are both Lebesgue measurable and bounded.
The former follows from the fact that they may be uniformly approximated by simple functions, a direct
consequence of Lemma 1 and the latter also. 

Convergence follows from the Lebesgue dominated convergence theorem.
This holds since the sequence $Y_n$ is uniformly bounded, by a constant so the sequence is dominated.
Second $Y$ is continuous a.e. with pointwise convergence holding at points of continuity, as a consequence
of convergence in $D$ see \cite{Billingsley68}.
\qed \end{proof}
We now proceed to prove the following Theorem.

\begin{theorem}\label{thm_weakconvergence}
With the topology induced by the previous metric in $\re^\infty$, we have that the sequence ${\bf T}_{N}$ converges in distribution to the sequence ${\bf I}$
\begin{equation}\label{eqn_thmfin}
{\bf T}_N \Rightarrow {\bf I}
\end{equation}
as $N\to\infty$.
\end{theorem}
\begin{proof}
In order to do so we use Theorem 4.2 of \cite{Billingsley68}. Suppose that there is a metric space ${\mathcal S}$ with metric $\rho_0$ and sequences  
${\bf T}_{N, \epsilon}$, ${\bf I}_\epsilon$ and ${\bf T}_N$ all lying in ${\mathcal S}$ such that the following conditions hold,
\begin{eqnarray}\label{eqn_weakcond} 
 {\bf T}_{N, \epsilon} & \Rightarrow & {\bf I}_\epsilon \\
 {\bf I}_\epsilon & \Rightarrow & {\bf I} \nonumber
\end{eqnarray}
together with the further condition that given arbitrary $\eta > 0$,
\begin{equation}\label{eqn_probmetric}
\lim_{\epsilon \rightarrow 0} \limsup_{N \rightarrow \infty} 
\prob{ \rho_0({\bf T}_{N, \epsilon}, {\bf T}_N) \geq \eta } = 0.
\end{equation}
Then it holds that ${\bf T}_N \Rightarrow {\bf I}$. First, we define ${\bf I}_\epsilon$ using a realisation of the Brownian Bridge
$$
I_{r,\epsilon} := \big[ \Wc_{\varphi_r}(\psi) \log 2(1- \cos \psi) \big]^{2\pi - \epsilon}_{\epsilon}
- \int_\epsilon^{2\pi-\epsilon} \Wc_{\varphi_r}(\psi)  \frac{\sin \psi}{1 - \cos \psi} d \psi.
$$
The definition of the other sequence is more involved and so we defer it for a moment. We have shown that the limit integrals exist a.s. and so we only need to show that the 
first term converges to 0. Since $\log \big[2 \lb 1 - \cos \psi \rb\big] = O(\log \epsilon)$ 
when $\epsilon$ is small and in a neighbourhood of 0 and $2\pi$, we may invoke the 
Levy modulus of continuity, wrapped around at $2\pi$ to obtain that this term is $O( a(\epsilon)\log \epsilon)$.
Hence, coordinate convergence of the integrals holds so that
$$
\Big| I_{r,\epsilon} + \int_\epsilon^{2\pi-\epsilon} \Wc_{\varphi_r}(\psi)  \frac{\sin \psi}{1 - \cos \psi} d \psi \Big| \Rightarrow 0
$$
and it follows that ${\bf I}_\epsilon \Rightarrow {\bf I}$ as $\epsilon\to 0$, since the sign of the integral is immaterial. We have thus demonstrated the second condition of (\ref{eqn_weakcond}). Next we proceed by rewriting $T_N(\varphi_r)$ in terms of the empirical distribution function $F_N:[0,2\pi]\to [0,1]$ determined by
$$
F_N(\psi) := \frac{\# \lc \theta_q : 0\leq \theta_q \leq \psi \rc}{N}.
$$
By the definition of $F_{N}$ and the Lebesgue--Stieljes integral
\begin{eqnarray*}
T_N(\varphi_r) & = & \sqrt{N} \int_0^{2\pi} \log \big(2(1 - \cos(\varphi_r-\psi))\big) dF_N(\psi) \\
               & = & \sqrt{N} \int_0^{2\pi} \log \big(2(1 - \cos \tp)\big) dF_{N,\varphi_r}(\tp) 
\end{eqnarray*}
where the change of variables $\tp = \psi - \varphi$ has been made.  For $\psi \in [0,2\pi)$ we define $F_{N,\varphi}(\psi)$ as the ``cycled'' empirical distribution function of $F_{N}$ by

\begin{equation*}
F_{N,\varphi}(\psi) :=
\begin{cases}
\frac{\#\lc \varphi \leq \theta_q < \varphi+\psi \rc }{N}, & \text{if } \varphi \leq \psi < 2\pi-\varphi, \\ 
F_{N,\varphi}(2\pi) + \frac{\#\lc 0 \leq \theta_q \leq \psi-2\pi+\varphi \rc}{N},& \text{if } 2\pi-\varphi \leq \psi < 2\pi.
\end{cases}
\end{equation*}

To define the sequence $T_{N,\epsilon}(\varphi_r)$, we split the integral into two parts 
as in $\int^{2\pi - \epsilon}_\epsilon$ and $\int_0^\epsilon + \int_{2\pi - \epsilon}^{2\pi}$
and then use  integration by parts on the first part, which yields the expression,
\begin{eqnarray}
 T_{N,\epsilon}(\varphi_r) & := & %\sqrt{N} \int_{\epsilon}^{2\pi - \epsilon} 
%\log 2 (1 - \cos \psi ) dF_{N,\varphi_r}(\psi) \\
%& = & 
\sqrt{N} \times \lb \ls \lb F_{N,\varphi_r}(\psi) - \frac{\psi}{2\pi} \rb \log 2(1 - \cos \psi) \rs_\epsilon^{2\pi-\epsilon} \rb \\
\label{eqn_TNepsdefn}
 & - & \sqrt{N}  \int_\epsilon^{2\pi - \epsilon} \lb F_{N,\varphi_r}(\psi) - \frac{\psi}{2\pi} \rb\frac{\sin \psi}{(1 - \cos \psi)} d\psi \nonumber .
\end{eqnarray}
For later use, we define
$$
W_{N,\varphi} := \sqrt{N} \lb  F_{N,\varphi}(\psi) - \frac{\psi}{2\pi} \rb.
$$
This is not quite equal to the original sum, since
$$
\int_0^{2\pi}  \log \big(2 \lb 1 - \cos \psi \rb\big) d \psi = 0
$$
so that the $\psi$ terms do not give 0 but rather cancel with $\mu_\epsilon$ to be defined in  a moment. We express the remainder as a sum, noting that we must include the mean, which
is by symmetry,
\begin{equation}
\mu_\epsilon := \frac{2}{2\pi} \int_0^\epsilon \log \big(2\lb 1 - \cos \psi \rb\big) d \psi = \frac{2}{\pi} \lb \epsilon \log \epsilon - \epsilon +o(\epsilon) \rb.
\end{equation}
Define $S_\epsilon(\varphi) := \lc \theta_q: \theta_q \in [\varphi-\epsilon,\varphi+\epsilon] \rc$ so that the sum may be written as
\begin{equation}
Z_{N,\epsilon}(\varphi_r) := \frac{1}{\sqrt{N}}
\sum_{\theta_q \in S_\epsilon(\varphi_r)}
 \log \big( 2(1 - \cos(\varphi_r-\theta_q))\big) - \sqrt{N} \mu_\epsilon.
\label{eqn_Zvarepsdef}
\end{equation}
Denote the corresponding sequence as ${\bf Z}_{N,\epsilon}$. Taking expectations we thus find that
$$
\expect{Z_{N,\epsilon}(\varphi_r)} = \frac{N}{\sqrt{N}} \int_{-\epsilon}^\epsilon \log \big( 2 (1-\cos\psi) \big) d\psi - \sqrt{N} \mu_\epsilon = 0
$$
is a sequence of random variables with 0 mean. We finally write,
\begin{equation}
{\bf T}_N = {\bf T}_{N,\epsilon} + {\bf Z}_{N,\epsilon}.
\label{eqn_seqdiff}
\end{equation}
We now proceed to demonstrate the first condition of (\ref{eqn_weakcond}), namely that, ${\bf T}_{N, \epsilon} \Rightarrow  {\bf I}_\epsilon$. The random variable $T_{N, \epsilon}(\varphi_r)$ is a functional of an empirical distribution and therefore of a process
lying in $D[0,2\pi]$. 
Define the random sequence $J_\epsilon$ for $f \in D[0,2\pi]$ and $f(0) = f(2\pi) = 0$ with the component term,
\begin{equation}
J_{\epsilon,r}(f) = \int_\epsilon^{2\pi - \epsilon} f_{\varphi_r}(\psi) 
  \frac{\sin \psi}{\lb 1 - \cos \psi \rb} d \psi -
\Big[ f_{\varphi_r}(\psi) \log \big[2(1 - \cos \psi)\big] \Big]_\epsilon^{2\pi-\epsilon}.
\label{eqn_Imap}
\end{equation} 
It is well known that, $W_{N,0} \Rightarrow \Wc$ in $D$, which implies that $W_{N,\varphi_r} \Rightarrow \Wc_{\varphi_r}$ as $N\to\infty$ for all $r$. The result follows by showing that $J_\epsilon$ defines a measurable mapping $J_\epsilon:D[0,2\pi] \rightarrow \re^\infty$ in $D[0,2\pi]$. Since 
$$
J_{\epsilon,r}(W_N) = T_{N,\epsilon}(\varphi_r),
$$ 
we may therefore apply Theorem 5.1 and Corollary 1 of \cite{Billingsley68} which states that if $W_N \Rightarrow \Wc$ then 
$J_\epsilon(W_N) \Rightarrow J_\epsilon(\Wc)$, (and hence ${\bf T}_{N,\epsilon} \Rightarrow {\bf I}_\epsilon$) provided that we verify 
\begin{equation}
\prob{ \Wc \in D_{J_\epsilon}} = 0.
\label{eqn_probdiscontinuity}
\end{equation}
To deal with the measurability question, we first observe that the coordinate maps are measurable and since $\sin \psi/(1-\cos \psi)$ is continuous on $[\epsilon, 2\pi - \epsilon]$, it follows by Lemma \ref{lem_Dctsmap} that $J_{\epsilon,r}$ is measurable for each $r$ and hence so is the sequence mapping $J_\epsilon$. Again by Lemma \ref{lem_Dctsmap} the sequence of integrals convergences with respect to $\rho_0$. This leaves
only the final term. However, since the limit $\Wc$ is almost surely continuous it follows almost surely that
\begin{eqnarray*}
f_{\varphi_r}(\epsilon) & \rightarrow & \Wc_{\varphi_r}(\epsilon) \\
f_{\varphi_r}(2\pi - \epsilon) & \rightarrow & \Wc_{\varphi_r}(2\pi-\epsilon) 
\end{eqnarray*}
for each $r$ if $f \rightarrow \Wc$ in $D[0,2\pi]$. Thus the corresponding sequence converges 
with respect to $\rho_0$ also and so (\ref{eqn_probdiscontinuity}) holds. The proof of the first condition is concluded. 

It remains to demonstrate (\ref{eqn_probmetric}). Here we use the union bound and Chebyshev's inequality. This is because the various $Z_{N,\epsilon}(\varphi_r)$ in the sequences are dependent, as
they are determined via the same $\theta_q$. Nevertheless they are of course themselves the sum of i.i.d. 
random variables. In determining the variance,
we may work with $\varphi_r=0$ without loss of generality. The variance
of one of the i.i.d. summands in (\ref{eqn_Zvarepsdef}) is
\begin{equation}
\sigma^2_\epsilon := \frac{2}{2\pi}  \int_0^\epsilon  \log^2 \big(2(1 - \cos\psi)\big)  d\psi - \mu_\epsilon^2 < \infty.
\end{equation}

Since for small $\epsilon > 0$ we have $\log \big(2(1 - \cos\psi)\big) = O(2\log \psi) + o(\psi)$, the integral
is $\sigma_\epsilon = O(\epsilon \log^2 \epsilon)$, as the integral of $\log^2 x$ is $x \log^2 x - 2x \log x + 2x$,
it follows that $\sigma^2_\epsilon \rightarrow 0$ as $\epsilon \rightarrow 0$. Now fix $\eta >0$. By definition of $\rho_0$ and from (\ref{eqn_seqdiff}), we obtain
$$
\rho_0({\bf T}_{N,\epsilon} , {\bf T}_N ) = \sum_{r=0}^\infty \frac{\abs{Z_{N,\epsilon}(\varphi_r)}}{1 + 
\abs{Z_{N,\epsilon}(\varphi_r)}} 2^{-r}.
$$
Let $R_\eta$ be such that $\sum_{r=R_\eta+1}^\infty 2^{-r} < \eta/2$. Now we apply the union bound to the
remaining $R_\eta+1$ summands to obtain
\begin{eqnarray}\label{eqn_chebunionbnd}
\prob{\sum_{r=0}^{R_\eta} \frac{\abs{ Z_{N,\epsilon}(\varphi_r)}}{1 + 
\abs{Z_{N,\epsilon}(\varphi_r)}} 2^{-r} \geq \eta/2 } &\leq &\sum_{r=0}^{R_\eta} \prob{\abs{Z_{N,\epsilon}(\varphi_r)} 2^{-r} \geq \frac{\eta}{2 \lb R_\eta +1 \rb}} \nonumber \\
& \leq & \sum_{r=0}^{R_\eta} \sigma^2_\epsilon \frac{4 \lb R_\eta +1 \rb ^2}{\eta^2 2^{2r}} \nonumber \\
& \leq & \sigma^2_\epsilon \frac{16 \lb  R_\eta +1 \rb ^2}{3\eta^2}.
\end{eqnarray}
Hence,
$$
\limsup_{N\to\infty} \prob{\rho_0({\bf T}_{N,\epsilon} , {\bf T}_N ) > \eta } \leq \frac{16(R_\eta + 1)^2 \sigma^2_\epsilon}{3\eta^2}
$$
and the RHS goes to 0 as $\epsilon\to 0$, for each $\eta > 0$. Thus we obtain (\ref{eqn_probmetric}) as 
required. We have therefore verified all conditions; and Theorem \ref{thm_weakconvergence} is proved.
\qed \end{proof}

We are now in a position to state our main result for an upper bound on the minimum eigenvalue of a random Vandermonde matrix.

\begin{theorem} \label{thm_minbridge}
Let $\lambda_1(N)$ be the minimum eigenvalue of the square $N\times N$ matrix ${\bf V}^{*}{\bf V}$. We further assume that the phases $\theta_{1},\ldots,\theta_{N}$ are i.i.d. and drawn accordingly to the uniform distribution.
Then 
\begin{equation}
\lambda_1(N) \leq 2N^{2}\exp\big(-\sqrt{N}T_N^{*}\big)
\label{eqn_minupbnd}
\end{equation}
where $T_N^* := \limsup_{r\to\infty} T_N(\varphi_r)$. 
%where $T_N^* \Rightarrow W^*$ and $W^* \in \re$ is defined as a functional of a Brownian bridge
%as in (\ref{eqqmax}). 
Moreover, given $a > 0$,
\begin{equation}
\liminf_{N\to\infty} \prob{\lambda_1(N) \leq 2N^2 \exp\big(-\sqrt{N}a\big)} \geq \prob{I^* > a}.
\end{equation}
\end{theorem}

\begin{proof}
Using the definition of $T_N(\varphi_r)$ in (\ref{eqn_TNdefn}) we obtain
$$
\lambda_1(N) \leq 4N^2 \exp\big(-\sqrt{N}T_N(\varphi_r)\big).
$$
Since this equation holds over every $r$ it follows that
\begin{equation}
\lambda_1(N) \leq 4N^2 \exp \big(-\sqrt{N}\limsup_{r\to\infty} T_N(\varphi_r)\big).
\end{equation}
Since this holds for all $N \in \NN$, we obtain (\ref{eqn_minupbnd}), which is the first part of the Theorem. Now define the random variable
$$
L_N := -\frac{1}{\sqrt{N}} \log \frac{\lambda_1(N)}{4N^2}
$$
and further, for any given $R$, define $T_{N,R} := \max_{r \leq R} \lc T_N(\varphi_r) \rc $, and similarly
define $I^*_R$. Then for any fixed $R$ and $a > 0$
$$
\prob{L_N > a} \geq \prob{T_{N,R} > a}
$$
by definition of $T_{N,R}$. By weak convergence, since the set is open, and by Theorem 2.1 of \cite{Billingsley68}, we obtain
$$
\liminf_{N\to\infty} \prob{L_N > a } \geq \liminf_N \prob{T_{N,R} > a} \geq \prob{ I^*_R > a}
$$
as a consequence of Theorem \ref{thm_weakconvergence}. Finally, by almost sure continuity it holds that $I_{R}^{*} \to I^{*}$ almost surely,
and so by the monotone convergence theorem we see that $\mathbb{P}(I^* > a)= \lim_{R\to\infty} \mathbb{P}(I^*_R > a)$, which implies our result.
\qed \end{proof}

\subsection{Analytical and Combinatorial Construction}
\par In this Section, we present an analytical and elementary argument for the upper bound of the minimum eigenvalue. Let $z_1,z_2,\ldots,z_N$ be complex numbers on the unit circle and let $P(z)=\prod_{i=1}^{N}{(z-z_i)}$ be the polynomial with these roots. We want to estimate $\max_{|z|=1}{|P(z)|}$ when the roots $\{z_{i}\}_{i=1}^{N}$ are i.i.d. uniformly distributed random variables on the unit circle.   

\begin{lemma}\label{lemma11}
Given $P(z)$ as before there exists $|w|=1$ such that $|P(w)P(-w)|=1$.
\end{lemma}

\begin{proof}
Consider the function $\Psi(z)=\log|P(z)|+\log|P(-z)|$. This function is continuous except at the values $\{z_1,-z_1,\ldots,z_N,-z_N\}$ where it has a vertical asymptote going to $-\infty$. Therefore, we can consider this function as a continuous function from the unit circle to $[-\infty,\infty)$ with the usual topology. On the other hand, it is clear that $\int_{|z|=1}{\Psi(z)}=0$. Therefore, there exist $w$ such that $\Psi(w)=0$ and hence $|P(w)||P(-w)|=1$.
\qed \end{proof}

\par Consider the following construction. We first randomly choose the points $\{z_{i}\}_{i=1}^{N}$ and consider the set of pairs $\mathcal{P}:=\{(z_1,-z_1),\ldots,(z_N,-z_N)\}$. Note that changing $z_i$ to $-z_i$ does not affect the value of the point $w$ in the previous Lemma. Hence the set $\mathcal{P}$ determines the point $w$. Now we fix this point and consider $\alpha_{i}:=|w-z_i|$ and $\beta_{i}:=|w+z_i|$. Since $|P(w)P(-w)|=1$, we see that $\prod_{i=1}^{N}{\alpha_i \beta_i}=1$. It is also clear that $\beta_i = \sqrt{4-\alpha_i^{2}}$. 

\par Let $y$ be the random variable defined as $y:\{1,-1\}^{N}\to\R$
$$
y(v_1,v_2,\ldots,v_N) = \sum_{i=1}^{N}{v_{i}\log(\alpha_i/\beta_i)}
$$
taking signs i.i.d. at random with probability $1/2$. It is not difficult to see that $\mathbb{E}(y)=0$, where the average is taken over the set $\{1,-1\}^{N}$. Note that 
\begin{eqnarray*}
y(v_1,\ldots,v_N) & = & \sum_{i=1}^{N}{v_{i}(\log(\alpha_i)-\log(\beta_i))}\\ 
& = &\log|P_{(v_1,\ldots,v_N)}(w)|-\log|P_{(v_1,\ldots,v_N)}(-w)|
\end{eqnarray*}
where $P_{(v_1,\ldots,v_N)}(z)$ is the polynomial with roots $v_{i}z_{i}$
\begin{equation}
P_{(v_1,\ldots,v_N)}(z) = \prod_{i=1}^{N}{(z-v_{i}z_{i})}
\end{equation}
and $w$ is as in Lemma \ref{lemma11}. Since $|P_{(v_1,\ldots,v_N)}(-w)|=|P_{(v_1,\ldots,v_N)}(w)|^{-1}$ we see that
\begin{equation}\label{fund}
y(v_1,\ldots,v_N) = \log|P_{(v_1,\ldots,v_N)}(w)|^{2}.
\end{equation}

\begin{theorem} Let $\gamma:=\log\Big(\frac{\cos(\pi/8)}{\sin(\pi/8)}\Big)$. For every $\epsilon>0$ the following holds 
$$
\mathbb{P}\Big(|y(v_1,\ldots,v_N)|\geq \gamma\sqrt{\pi}\epsilon\sqrt{N}/2\Big)\geq 1-\epsilon.
$$
\end{theorem}

\begin{proof}
By changing $z_i$ to $-z_i$ if necessary, we can assume without loss of generality that $\alpha_i\geq\beta_i$. The point $w$ is equal to $w=e^{i\theta}$ for some phase $\theta$ in $[0,2\pi)$. Let $\mathcal{A}$ be the set 
$$
\mathcal{A}:=\Big\{e^{i\phi}\,\,:\,\,\phi\in [\theta+\pi/4,\theta+3\pi/4]\cup [\theta-3\pi/4,\theta-\pi/4]\Big\}.
$$
The total length of the set $\mathcal{A}$ is $\pi$ and hence the probability of random point $z_{i}$ to belong to $\mathcal{A}$ is equal to $1/2$. On the other hand, it is easy to see that if $z_i$ belongs to the complement of $\mathcal{A}$ and since by assumption $\alpha_i\geq \beta_i$ we see that 
$$
\log(\alpha_i/\beta_i)\geq \log\Bigg(\frac{\cos(\pi/8)}{\sin(\pi/8)}\Bigg)=:\gamma \approx 0.8814.
$$ 
Let us order the values of $\log(\alpha_i/\beta_i)$ in increasing order. Up to a re--numeration we see that 
$$
0 \leq\ldots\leq \log(\alpha_{m}/\beta_{m})\leq  \gamma \leq \log(\alpha_{m+1}/\beta_{m+1})\leq\ldots\leq \log(\alpha_N/\beta_N).
$$
The value of $m$ is a random variable that converges almost surely to $N/2$ as $N\to\infty$. Without loss of generality and for notation simplicity, we take $m=N/2$, however, and as the argument shows this is not strictly necessary. Let $\epsilon>0$ and let us consider the value $\sup_{I}{\mathbb{P}(y(v_1,\ldots,v_N)\in I)}$ where $I$ ranges over all the closed intervals of length $\gamma\sqrt{\pi}\epsilon\sqrt{N}$ in the real line. By applying the Littlewood--Offord Theorem, discussed in the preliminaries Section, we see that
\begin{equation}\label{eqq1}
\sup_{I}\Big\{\mathbb{P}\Big(y(v)\in I\,\,|\,\,\text{conditioning on $v_{i}$ for $i\leq \lfloor N/2 \rfloor$}\Big)\Big\} = \epsilon + o(N^{-1/2}) 
\end{equation}
for $N$ sufficiently large. On the other hand,
$$
\sup_{I}\Big\{\mathbb{P}\big(y(v_1,\ldots,v_N)\in I\big)\Big\}
$$
is equal to 
$$
\frac{1}{2^{N/2}}\sum_{(v_{1},\ldots,v_{\lfloor N/2 \rfloor})}{\sup_{I}\Big\{\mathbb{P}\Big(y(v)\in I\,\,|\,\,\text{cond. on the first $v_{i}$}\Big)\Big\}}\leq\epsilon
$$
where the last inequality follows from (\ref{eqq1}). In particular, taking $I$ to be the interval $I =[-\gamma\sqrt{\pi}\epsilon\sqrt{N}/2, \gamma\sqrt{\pi}\epsilon\sqrt{N}/2]$ we conclude that
$$
\mathbb{P}\Big(|y(v_1,\ldots,v_N)|\geq \gamma \sqrt{\pi}\epsilon\sqrt{N}/2\Big)\geq 1-\epsilon.
$$
\qed \end{proof}

\begin{theorem}\label{randpoly}
Given $\epsilon>0$ we have that
\begin{equation}
\mathbb{P}\,\Big(\max_{|z|=1}|P(z)|^2\geq \exp(\gamma\sqrt{\pi}\epsilon\sqrt{N}/2)\Big)\geq 1-\epsilon
\end{equation}
for $N$ sufficiently large.
\end{theorem}

\begin{proof}
As done before, we start by randomly generating $n$ pairs of diametrically opposite points $(z_{i},-z_{i})$ and find $w$ as in Lemma \ref{lemma11}. Finally, fix the $z_i$ by the $N$ independent coin flips $(v_{1},\ldots,v_{N})$ and condition on this event. We observed before that $|P_{(v_1,\ldots,v_N)}(-w)|=|P_{(v_1,\ldots,v_N)}(w)|^{-1}$. Therefore,
$$
\log|P_{(v_1,\ldots,v_N)}(w)|=-\log|P_{(v_1,\ldots,v_N)}(-w)|.
$$
Let $a(w)$ be 
\begin{eqnarray*}
a(w) & = & \log|P_{(v_1,\ldots,v_N)}(w)|-\log|P_{(v_1,\ldots,v_N)}(-w)|\\
& = & \log|P_{(v_1,\ldots,v_N)}(w)|^2.
\end{eqnarray*}
Then by the previous Theorem and (\ref{fund}) we see that
$$
\mathbb{P}\Big(|a(w)| \geq \gamma\sqrt{\pi}\epsilon\sqrt{N}/2\Big)\geq 1-\epsilon.
$$
Since $a(-w)=-a(w)$ we clearly see that
$$
\mathbb{P}\big(\max\big(a(w),a(-w)\big)\geq \sqrt{\pi}\epsilon\sqrt{N}/2\big)\geq 1-\epsilon.
$$
Therefore, 
\begin{equation}
\mathbb{P}\,\Big(\max_{|z|=1}{\max \{\log|P_{(v_1,\ldots,v_N)}(z)|^2,\log|P_{(v_1,\ldots,v_N)}(-z)|^2 \} } \geq \sqrt{\pi}\epsilon\sqrt{N}/2\Big) \geq 1-\epsilon. 
\end{equation}
Now removing the conditioning on the pairs $\{z_1,-z_1,\ldots,z_N,-z_N\}$ we see that,
\begin{equation}
\mathbb{P}\,\Bigg(\max_{|z|=1}|P(z)|^2\geq \exp\big(\sqrt{\pi}\epsilon\sqrt{N}/2\big)\Bigg)\geq 1-\epsilon
\end{equation}
for $N$ sufficiently large.
\qed \end{proof}
Since we already saw in Lemma \ref{lem_polyineq} that
$$
\lambda_{1}(N)\leq \frac{4N^2}{\max_{|z|=1}|P(z)|^2}
$$
the following Theorem follows immediately.

\begin{theorem}\label{main_comb}
Given $\epsilon>0$ we have that
\begin{equation}
\mathbb{P}\,\Big(\lambda_1(N)\leq 4N^2\exp(-\sqrt{\gamma \pi}\epsilon\sqrt{N}/2)\Big)\geq 1-\epsilon
\end{equation}
for $N$ sufficiently large.
\end{theorem}

\section{Numerical Results}\label{sec_numer}

In this Section we present some numerical results for the behavior near the origin of the limit probability distribution of ${\bf V}^{*}{\bf V}$, and for the minimum eigenvalue $\lambda_1$. Let ${\bf V}$ be a square $N\times N$ random Vandermonde matrix with phases $\theta_{1},\theta_{2},\ldots, \theta_{N}$, which are i.i.d. random variables uniformly distributed on $[0,1]$. We know that the empirical eigenvalue distribution of ${\bf V}^{*}{\bf V}$ converges as $N\to\infty$ to a probability measure $\mu$. One question that we would like to address is: does the measure $\mu$ have an atom at zero?

Let $\{\lambda_{i}\}_{i=1}^{N}$ be the eigenvalues of ${\bf V}^{*}{\bf V}$. Given $\epsilon>0$ let us denote by $G_{N}(\epsilon)$ the average number of eigenvalues less than or equal to $\epsilon$, i.e.,
$$
G_{N}(\epsilon):=\frac{1}{N} \mathbb{E}\Big( \Big| \big\{\lambda_{i}<\epsilon\,:\,i=1,\ldots,N\big\} \Big| \Big).
$$
Therefore, if there is an atom at zero for the measure $\mu$ with mass $\mu\{0\}=\beta$, the following holds
\begin{equation}
\inf_{\epsilon>0}\,\liminf_{N\to\infty}\,G_{N}(\epsilon)=\beta.
\end{equation}

\begin{figure}[!Ht]
  \begin{center}
    \includegraphics[width=10cm]{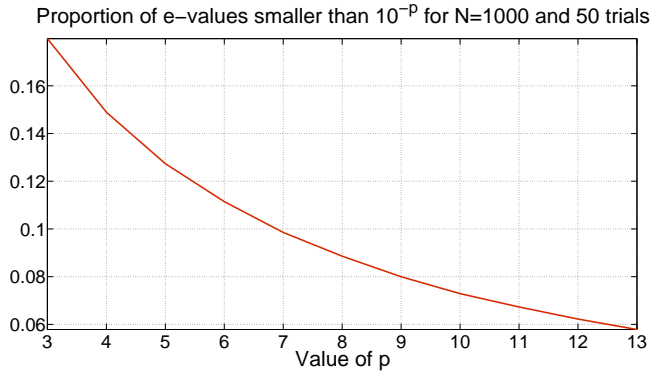}
    \caption{Graph of the average proportion of eigenvalues smaller than $10^{-p}$ as a function of $p$ for $N=1000$.}
    \label{atom}
  \end{center}
\end{figure}

In Figure \ref{atom}, we plot $G_{N}(10^{-p})$ as a function of $p$ for $N=1000$. These graphs suggest that if there is an atom, its mass has to be relatively small. Further simulations suggest the absence of an atom at zero. However, at the moment, we are unable to prove this result.  

%\subsection{Results for the Minimum Eigenvalue}
%Our theoretical results show that the minimum eigenvalue is decreasing extremely fast. In Figure \ref{fig_dyadicplt}, we show a realization of $T_N$ for a subset of the %dyadic fractions
%and a large value of $N$, which indicates its qualitative behaviour. The maximum value over the values of $\varphi$ selected is around 4. In Figure \ref{fig_TNmaxplt}, we %have obtained an empirical pdf for the above maximum.
%\begin{figure}[!Ht]
%  \begin{center}
%    \includegraphics[width=10cm]{Tvarphi_15R20N.png}
%    \caption{$T_N(\varphi)$ for $2^{15}$ dyadic rationals with $N = 10^6$.}
%    \label{fig_dyadicplt}
%  \end{center}
%\end{figure}
%\begin{figure}[!Ht]
%  \begin{center}
%    \includegraphics[width=10cm]{Tmax_10R_32N.png}
%    \caption{$\max_r T_{N,R}(\varphi_r)$ for $2^{15}$ dyadic rationals with $N = 32$ and 100,000 trials.}
%    \label{fig_TNmaxplt}
%  \end{center}
%\end{figure}
%%In Figure \ref{fig_mineig}, we show a rescaled version for the actual minimum eigenvalue. The right tail has been truncated so that the minimum value is at least $10^{-15}$. The results indicate %that the distribution of the minimum eigenvalue lies to the right of that for $T_N$. 
%We show the results for $\max_r T_{N,R}(\varphi_r)$ when $R=2^{15}$.
%%\begin{figure}[!Ht]
%%  \begin{center}
%%    \includegraphics[width=10cm]{mineig32.png}
%%    \caption{$L_N = -\frac{2}{\sqrt{N}} \log \frac{\lambda_1}{2N^2}$ with $N=32$ and 100,000 trials.}
%%    \label{fig_mineig}
%%  \end{center}
%%\end{figure}
\begin{figure}[!Ht]
  \begin{center}
    \includegraphics[width=10cm]{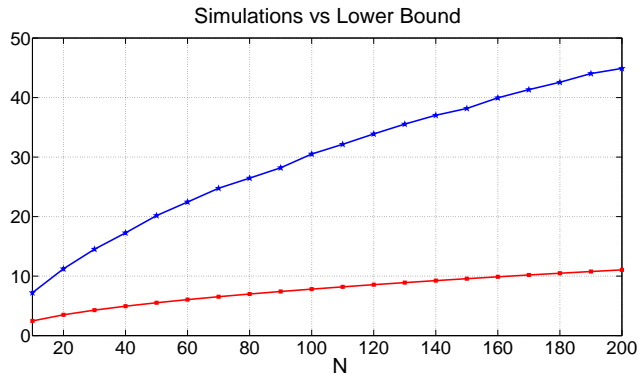}
    \caption{Graphs of the average of $2\log \max_{|z=1|}|P(z)|$ (blue) and $\sqrt{\gamma \pi}\epsilon\sqrt{N}/2$ (red) as a function of $N$ where $2\log \max_{|z=1|}|P(z)|$ was averaged over $1000$ realizations.}
    \label{ff1}
  \end{center}
\end{figure}

Finally, we present some numerical results for the behavior of the maximum of a random polynomial on the unit circle in the context of Theorem \ref{randpoly}. In Figure \ref{ff1}, we show the graphs of $2\log \max_{|z=1|}|P(z)|$ and $\sqrt{\gamma \pi}\epsilon\sqrt{N}/2$ as a function of $N$. This graph suggests that Theorem \ref{randpoly} could be slightly improved.

\section{Generalized Random Vandermonde Matrix}

In this Section we present a generalized version of the previously discussed random Vandermonde matrices. More specifically, consider an increasing sequence of integers $\{k_{p}\}_{p=1}^{\infty}$ and let $\{\theta_{1},\ldots,\theta_{N}\}$ be i.i.d. random variables uniformly distributed on the unit interval $[0,1]$. Let $\bf{V}$ be the $N\times N$ random matrix defined as
\begin{equation} 
V(p,q):=\frac{1}{\sqrt{N}}z_{q}^{k_p}
\end{equation} 
where $z_{q}:=e^{2\pi i\theta_{q}}$. Note that if we consider the sequence $k_p=p-1$ then the matrix $\bf{V}$ is the usual random Vandermonde matrix defined in  (\ref{eqn_Vandermondedefn}). We are interested in understanding the limit eigenvalue distribution for the matrices $\bf{X}:=\bf{V}\bf{V}^{*}$ and in particular their asymptotic moments. Let $r\geq 0$ and let us define the $r$--th asymptotic moment as 
\begin{equation}
m_{r}:=\lim_{N\to\infty} \mathbb{E}\Big( \mathrm{tr}_{N}({\bf X}^{r})\Big).
\end{equation}
These moments, as well as the limit eigenvalue distribution, depend on the sequence $\{k_{p}\}_{p=1}^{\infty}$. 

\begin{remark}
It is a straight forward calculation to see that $m_0=m_{1}=1$, $m_{2}=2$ and $m_3=5$ no matter what is the sequence $\{k_{p}\}_{p=1}^{\infty}$. The first interesting case happens when $r$ is equal to 4. These is because $r=4$ is the first positive integer where there is a non--crossing partition, namely the partition $\rho=\{\{1,3\},\{2,4\}\}$.
\end{remark}

The next Theorem shows a combinatorial expression for the moments as well as the existence of the limit eigenvalue distribution.

\begin{theorem}
Let $\{k_{p}\}_{p=1}^{\infty}$ be an increasing sequence of positive integers. Then
\begin{equation}
m_r = \sum_{\rho\in \mathcal{P}(r)}{K_{\rho}}
\end{equation}
where $\mathcal{P}(r)$ is the set of partitions of the set $\{1,2,\ldots,r\}$ and 
\begin{equation}
K_{\rho}:=\lim_{N\to\infty}\frac{|S_{\rho,N}|}{N^{r+1-|\rho|}}
\end{equation}
where $|\rho|$ is the number of blocks of $\rho$ and 
\begin{equation}
S_{\rho,N}:=\Big\{(p_1,\ldots,p_r)\in\{1,2,\ldots,N\}^{r}\,:\,\sum_{i\in B_{j}}{k_{p_i}}=\sum_{i\in B_{j}}{k_{p_{i+1}}}  \Big\},
\end{equation}
where $B_{j}$ are the blocks of $\rho$. Moreover, there exists a unique probability measure $\mu$ supported in $[0,\infty)$ with these moments.
\end{theorem}

\begin{proof}

Given $r\geq 0$ then
\begin{eqnarray*}
\mathrm{tr}_{N}({\bf X}^{r}) & = & \frac{1}{N}\sum_{(p_1,\ldots,p_r)}{X(p_1,p_2)X(p_2,p_3)\ldots X(p_r,p_1)} \\
& = & \frac{1}{N^{r+1}}\sum_{(p_1,\ldots,p_r)} \sum_{(i_1,\ldots,i_r)} z_{i_1}^{(k_{p_1}-k_{p_2})}z_{i_2}^{(k_{p_2}-k_{p_3})}\ldots z_{i_r}^{(k_{p_r}-k_{p_1})}.
\end{eqnarray*}
The sequence $(i_1,i_2,\ldots,i_r)\in \{1,2,\ldots,N\}^{r}$ uniquely defines a partition $\rho$ of the set $\{1,2,\ldots,r\}$ (we denote this by $(i_1,\ldots,i_r)\mapsto \rho$)
where each block $B_{j}$ consists of the positions which are equal, i.e.,
$$
B_{j}=\{w_{j_1},\ldots,w_{j_{|B_j|}}\}
$$ 
where $i_{w_{j_1}}=i_{w_{j_2}}=\ldots=i_{w_{j_{|B_j|}}}$. Denote this common value by $W_{j}$. Then
\begin{equation}\label{genn}
\mathrm{tr}_{N}({\bf X}^{r}) = \frac{1}{N^{r+1}} \sum_{(i_1,\ldots,i_r)\mapsto \rho}\,\,\,\sum_{(p_1,\ldots,p_r)} \prod_{k=1}^{|\rho|}{z_{W_k}^{\sum_{i\in B_k}{(k_{p_i}-k_{p_{i+1}})}}}.
\end{equation}
Taking expectation on both sides we observe that 
$$
\mathbb{E}\Bigg(z_{W_k}^{\sum_{i\in B_k}{(k_{p_i}-k_{p_{i+1}})}}\Bigg)\neq 0
$$
if and only if $\sum_{i\in B_k}{(k_{p_i}-k_{p_{i+1}})}=0$. Let $S_{\rho,N}$ be the $r$-tuples $(p_1,\ldots,p_r)$ which solve the equations
\begin{equation}\label{ttt}
\sum_{i\in B_k}{k_{p_i}}=\sum_{i\in B_k}{k_{p_{i+1}}}
\end{equation}
for all the the blocks $k=1,2,\ldots,|\rho|$ and let $|S_{\rho,N}|$ be its cardinality. Let $K_{\rho}$ be defined as 
$$
K_{\rho}=\lim_{N\to\infty}\frac{|S_{\rho,N}|}{N^{r+1-|\rho|}}.
$$
Then it follows from (\ref{genn}) that 
$$
m_{r}=\sum_{\rho\in \mathcal{P}(r)}K_{\rho}.
$$
It is straight forward to see that the set of solutions of (\ref{ttt}) has $r+1-|\rho|$ free variables since one of the equations is redundant (the sum of all the equations is 0). Therefore, for every partition $\rho$ the value of $K_{\rho}$ satisfies $0\leq K_{\rho}\leq 1$. Then the moments are bounded by the Bell numbers $B_{r}=|\mathcal{P}(r)|$. Define, 
$$
\beta_{r}:=\inf_{k\geq r}{m_{k}^{\frac{1}{2k}}}\leq {B_{k}^{\frac{1}{2k}}}\leq \inf_{k\geq r}{\sqrt{k}}=\sqrt{r}.
$$
Hence, $\beta_{r}^{-1}\geq r^{-1/2}$ and therefore
$$
\sum_{r=1}^{+\infty}{\beta_{r}^{-1}}=+\infty.
$$
Therefore, by Carleman's Theorem \cite{Carleman} there exists a unique probability measure $\mu$ supported on $[0,+\infty)$ such that 
$$
m_{r}=\int_{0}^{+\infty}{t^{n}\,d\mu (t)}.
$$
In other words, the sequence $m_{r}$ is {\em distribution determining}. 
\qed \end{proof}

\begin{prop}
Let $\rho\in\mathcal{P}(r)$ then $K_{\rho}=1$ if and only if the partition $\rho$ is non--crossing.
\end{prop}
The proof of this results follows similarly to the one presented in \cite{GC02} for the sequence $k_{p}=p-1$ and we leave it as an exercise for the reader. 

\begin{example} Let $r=4$ and let $\rho=\{\{1,3\},\{2,4\}\}$. Then
$$
K_{\rho}=\lim_{N\to\infty}\frac{|S_{\rho,N}|}{N^3}
$$
where 
$$
S_{\rho,N}=\big\{(p_1,p_2,p_3,p_4)\in \{1,2,\ldots,N\}^{4}\,:\,k_{p_1}+k_{p_3}=k_{p_2}+k_{p_4}\big\}.
$$
For the case $k_{p}=p-1$ it was observed in \cite{GC02} that $K_{\rho}=2/3$. As a matter of fact, it is not difficult to see that $K_{\rho}$ is the volume of the polytope
$$
K_{\rho}=\mathrm{vol}\big(\{(x,y,z)\in [0,1]^3\,:\ 0\leq x+y-z\leq 1\}\big).
$$
This polytope is shown in Figure \ref{poly}. 
\begin{figure}[!Ht]
  \begin{center}
    \includegraphics[width=6cm]{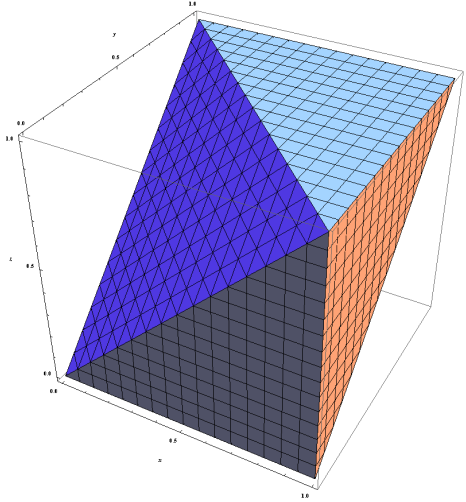}
    \caption{The polytope $(x,y,z)\in [0,1]^{3}$ such that $0\leq x+y-z\leq 1$.}
    \label{poly}
  \end{center}
\end{figure}
For the case $k_{p}=2^{p}$ we see that 
$$
S_{\rho,N}=\big\{(p_1,p_2,p_3,p_4)\in \{1,2,\ldots,N\}^{4}\,:\,2^{p_1}+2^{p_3}=2^{p_2}+2^{p_4}\big\}.
$$
For positive integers $\{a,b,c,d\}$ the equation $2^{a}+2^{b}=2^{c}+2^{d}$ holds if and only if $\{a,b\}=\{c,d\}$. Therefore, $|S_{\rho,N}|=2N^2$ and hence $K_{\rho}=0$.

\end{example}

The next Theorem shows that if $k_p=2^p$ then the limit eigenvalue distribution is the famous Marchenko--Pastur distribution. 

\begin{theorem}
Let $k_{p}=2^{p}$ then for every $r$ and $\rho\in\mathcal{P}(r)$ the coefficient $K_{\rho}=0$ if the partition is crossing. Hence
$$
m_{r}=|NC(r)|
$$
the number of non--crossing partitions and $\mu$ is the Marchenko--Pastur distribution
$$
d\mu(x)=\frac{1}{2\pi}\sqrt{\frac{4-x}{x}}\mathbf{1}_{[0,4]}.
$$
\end{theorem}

\begin{proof}
We already observed that $K_{\rho}=1$ iff the partition is non--crossing. Therefore, we need to show that for every crossing partition $K_{\rho}=0$. Let $\rho\in\mathcal{P}(r)$ be a crossing partition 
with blocks $\{B_1,B_2,\ldots,B_{|\rho|}\}$. Let $I$ be the set of indices such that for $i\in I$ the block $B_{i}$ does not cross any other block $B_{j}$. Then we can decompose $\rho$ as $\rho = \rho_1\cup \rho_{2}$ where $\rho_{2}=\cup_{i\in I}{B_{i}}$ is the union of all the non--crossing blocks. Then by the definition of $K_{\rho}$ we see that $K_{\rho}=K_{\rho_1}K_{\rho_2}$. Now we need to show that $K_{\rho_1}=0$. Up to a re-enumeration, if necessary, we see that $\rho_1\in\mathcal{P}(s)$ where $s\leq r$. By definition every block of $\rho_1$ crosses at least another block. For every $n$--tuples of positive integers $(a_1,\ldots,a_n)$ and $(b_1,\ldots,b_n)$ the equation
$$
2^{a_1}+2^{a_2}+\ldots+2^{a_n}=2^{b_1}+2^{b_2}+\ldots+2^{b_n}
$$
implies that $\{a_1,\ldots,a_n\}=\{b_1,\ldots,b_n\}$. Hence, every equation in $S_{\rho_1,N}$ eliminates at least two variables and therefore $|S_{\rho_1,N}|=O(N^{s+1-2|\rho_1|})$. This implies that
$$
K_{\rho_1}=\lim_{N\to\infty}\frac{|S_{\rho_1,N}|}{N^{s+1-|\rho_1|}}=0
$$ 
finishing the proof.
\qed \end{proof}

In Figure \ref{MP}, we see the histogram of the matrix $\bf{V}\bf{V}^{*}$ for the sequence $k_p=2^{p}$ and $N=100$ over $1000$ trials in comparison with the Marchenko--Pastur distribution. As it can be appreciated even for $N$ as small as 100 the two are not to far apart. In Figure \ref{classic}, we see the histogram of the eigenvalues of $\bf{V}\bf{V}^{*}$ for $k_p=p-1$ and $N=100$.
\begin{figure}[!Ht]
  \begin{center}
    \includegraphics[width=10cm]{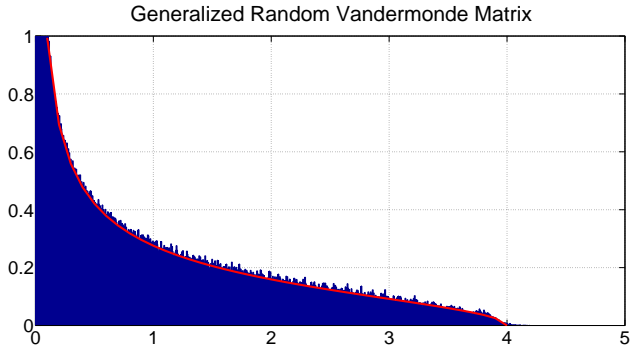}
    \caption{The blue graph is the histogram of eigenvalues of the matrix $\bf{V}\bf{V}^{*}$ for the sequence $k_p=2^{p}$ and $N=100$ over $1000$ trials. The red curve is the Marchenko--Pastur distribution.}
    \label{MP}
  \end{center}
\end{figure}
\begin{figure}[!Ht]
  \begin{center}
    \includegraphics[width=10cm]{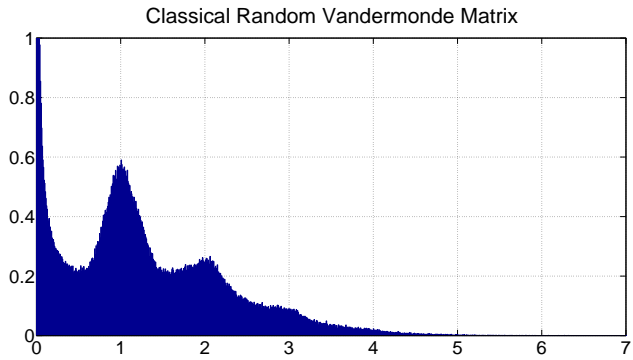}
    \caption{Histogram of the eigenvalues of matrix $\bf{V}\bf{V}^{*}$ for the sequence $k_p=p-1$ and $N=100$ over $1000$ trials.}
    \label{classic}
  \end{center}
\end{figure}
The case $k_{p}=p^2$ is an interesting one (as well as the cases $k_{p}=p^{a}$). At the moment we don't understand what is the limit eigenvalue distribution for this sequence. For instance, is it true that $K_{\rho}=0$ for every crossing partition? Is it true that $K_{\rho}=0$ for the partition $\rho=\{\{1,3\},\{2,4\}\}$? In a private communication with Prof. Carl Pomerance it was indicated that $|S_{\rho,N}|$ is of the order $O(N^2\log(N))$. However, we are not providing a proof of this fact. In Figure \ref{ffg}, we show the values of $|S_{\rho,N}|/N^3$ as a function of $N$ and we compare it with the case $k_{p}=2^{p}$. 
\begin{figure}[!Ht]
  \begin{center}
    \includegraphics[width=10cm]{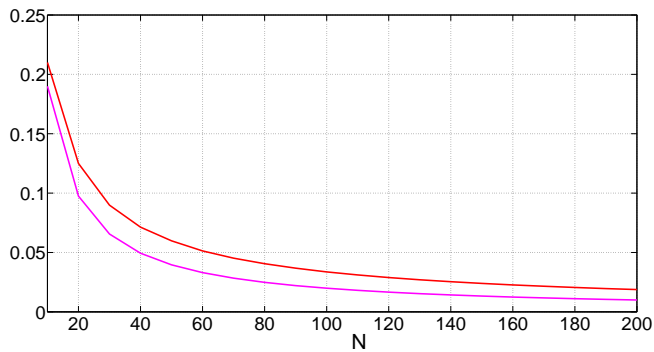}
    \caption{This figure shows $|S_{\rho,N}|/N^3$ for the sequence $k_{p}=p^2$ (red) and $k_{p}=2^{p}$ (magenta).}
    \label{ffg}
  \end{center}
\end{figure}

\end{document}